\let\ORIlabel\label
\let\ORIrefstepcounter\refstepcounter
\AddToHook{package/hyperref/before}{
   \let\label\ORIlabel 
   \let\refstepcounter\ORIrefstepcounter}
\documentclass[onefignum,onetabnum]{siamart220329}
\usepackage[normalem]{ulem}
\usepackage{lipsum,verbatim,tcolorbox}
\usepackage{wrapfig}
\usepackage{amsfonts}
\usepackage{graphicx}
\usepackage{multirow}
\usepackage{epstopdf}
\usepackage{algorithmic}
\usepackage{subcaption}
\usepackage{colortbl, hhline}

\ifpdf
  \DeclareGraphicsExtensions{.eps,.pdf,.png,.jpg}
\else
  \DeclareGraphicsExtensions{.eps}
\fi
\usepackage[textsize=footnotesize,textwidth=0.8in]{todonotes}
 \usepackage{pifont}
 
 \usepackage{mathabx}
 \newcommand{\cmark}{\ding{51}}%
\newcommand{\xmark}{\ding{55}}%

\newcommand{\bv}{\mathbf{v}}
\newcommand{\bu}{\mathbf{u}}
\newcommand{\ee}{\boldsymbol e}
\newcommand{\bz}{\boldsymbol z}
\newcommand{\R}{\mathbb R}
\newcommand{\N}{\mathbb N}

\newcommand{\Xscr}{\mathcal{X}}

\newcommand{\bxi}{\boldsymbol{\xi}}
\newcommand{\argmin}{\mathop{\mathrm{argmin}}}

\usepackage{amsmath,multirow}
\usepackage{algorithm}
\usepackage{algorithmic}
\usepackage{appendix}

\def \R {\mathbb{R}}
\def \N {\mathbb{N}}

\newsiamremark{remark}{Remark}
\newsiamremark{hypothesis}{Hypothesis}
\crefname{hypothesis}{Hypothesis}{Hypotheses}
\newsiamthm{claim}{Claim}
\newsiamthm{assumption}{Assumption}
\headers{Computing Equilibria via Nikaido-Isoda Functions}{ }

\title{On the Sampling-based Computation of Nash Equilibria under Uncertainty via the  Nikaido-Isoda Function}

\author{L. Marrinan\thanks{L. Marrinan is in the Department of Industrial \& Manufacturing Engineering, Pennsylvania State University (\email{lwm5431@psu.edu})} \and  U. V. Shanbhag\thanks{U. V. Shanbhag is in the Department of Industrial \& Operations Engineering, University of Michigan at Ann Arbor (\email{udaybag@umich.edu})} \and F. Yousefian\thanks{F. Yousefian is with the Department of Industrial \& Systems Engineering, Rutgers University (\email{farzad.yousefian@rutgers.edu})} 
\funding{This work was funded in part by in part by the ONR under grants N$00014$-$22$-$1$-$2589$ and N$00014$-$22$-$1$-$2757$, AFOSR
Grant FA9550-24-1-0259, and in part by the DOE under grant DE-SC$0023303$.}}

\usepackage{amsopn}

\makeatletter
\newcommand*{\addFileDependency}[1]{% argument=file name and extension
  \typeout{(#1)}% latexmk will find this if $recorder=0 (however, in that case, it will ignore #1 if it is a .aux or .pdf file etc and it exists! if it doesn't exist, it will appear in the list of dependents regardless)
  \@addtofilelist{#1}% if you want it to appear in \listfiles, not really necessary and latexmk doesn't use this
  \IfFileExists{#1}{}{\typeout{No file #1.}}% latexmk will find this message if #1 doesn't exist (yet)
}
\makeatother

\usepackage{thmtools}

\declaretheoremstyle[
  headfont=\normalfont\bfseries,
  bodyfont=\normalfont,
numberwithin=section,
spaceabove=8pt,
mdframed={
 linecolor=gray!20,
  skipabove=18pt,
  innerbottommargin=10pt, 
  backgroundcolor=gray!20,
  innerleftmargin=3pt,
  innerrightmargin=8pt}
]{mytheorem}

\ifpdf
\hypersetup{
  pdftitle={},
  pdfauthor={}
}
\fi

\def\bko{{\rm 1\kern-.17em l}}

\def\Dscr{{\mathcal D}}

\def\Kscr{{\mathcal K}}

\newcommand{\uu}{\mathbf{u}}

\def\be{\begin{enumerate}}
\def\ee{\end{enumerate}}

\def\argmin{\mathop{\rm argmin}}

 \newcommand{\remove}[1]{}

\newcommand{\pmat}[1]{\begin{pmatrix} #1 \end{pmatrix}}

\newcommand{\x}{{\mathbf{x}}}
\newcommand{\y}{{\mathbf{y}}}
\newcommand{\vv}{{\mathbf{v}}}

\def\Real{\mathbb{R}}

\def\argmin{\mathop{\rm argmin}}

\definecolor{britishracinggreen}{rgb}{0.0, 0.36, 0.15}

\newcommand{\uvs}[1]{\begin{color}{black}#1\end{color}}

\newcommand{\lm}[1]{\begin{color}{blue}#1\end{color}}
\newcommand{\fyy}[1]{\begin{color}{black}#1\end{color}}
 \usepackage{graphicx}

\usepackage{tikz}
\usetikzlibrary{decorations.pathreplacing,calc}

\begin{document}

\maketitle
\thispagestyle{empty}
\begin{center} {\em 
The authors would like to dedicate this article to Prof. Tamas Terlaky for both his deep and enduring contributions to the field of optimization as well as his leadership and mentorship.}
\end{center}

%%%%%%%%%%%%%%%%%%%%%%%%%%%%%%%%%%%%%%%%%%%%%%%%%%%%%%%%%%%%%%%%%%%%%%%%%%%%%%%%
\begin{abstract}  
    {We consider the computation of an equilibrium of a stochastic Nash
equilibrium problem, where the player objectives are assumed to be $L_0$-Lipschitz
continuous and convex given rival decisions with convex and closed
player-specific feasibility sets. To address this problem, we consider
minimizing a suitably defined value function associated with the Nikaido-Isoda
function. Such an avenue does not necessitate either monotonicity properties of the concatenated gradient map or potentiality requirements on the game but does require a suitable regularity requirement under which a  stationary point is a Nash equilibrium.  We design and analyze a sampling-enabled projected gradient
descent-type method, reliant on inexact resolution of a player-level
best-response subproblem. By deriving suitable Lipschitzian guarantees on the
value function,  we derive both asymptotic guarantees for the sequence of
iterates as well as rate and complexity guarantees for computing a stationary
point by appropriate choices of the sampling rate and inexactness sequence.}
\end{abstract}

\section{Introduction}
Noncooperative game-theory provides a foundation for the
analysis and computation of equilibria and such models have been used to capture conflict between self-interested parties in a range of settings   arising in operations research, engineering, economics, among other disciplines. It has  been
the subject of a collection of influential monographs including~\cite{facc2010,fudenberg98theory}. A particularly important
question considered in such settings lies in the computation of a Nash
equilibrium in $N$-player games~\cite{nash50equilibrium}. We focus on precisely such a question when
player objectives are expectation-valued and satisfy suitable convexity and smoothness
properties. More formally, consider a set of self-interested players (agents), each of whom
minimizes her  objective, given rival decisions.  Let the number of players be
denoted by $N$, indexed by $\nu \in \{1, \dots, N\},$ where for any $\nu \, \in
\, \left\{\, 1, \dots, N\,\right\}$, player $\nu$'s strategy is denoted by
$\x^{\nu} \in \Xscr^{\nu} \subseteq \Real^{n_{\nu}}$ and $\sum_{\nu=1}^N n_{\nu} = n$. 
%Let $\left(\x^{1}, \dots, \x^{N} \right)$ be the vector resulting from concatenating all of the players decision variables. To emphasize a certain player's decision variables, it will at times be convenient to write
%$\left( \x^{\nu},\x^{-\nu} \right)$ 
{Further,  let $\x^{-\nu}$ denote the strategy-tuple of all players, other than
player $\nu$, i.e.{,} $\x^{-\nu} \equiv (\x^j)_{\nu \ne j=1}^N$. Importantly, $\x
= \left(\, \x^{\nu},\x^{-\nu} \, \right)$ is not a re-cycling or rearrangement
of the vector $\x$, but merely emphasizes the coordinates that correspond to
player $\nu$.  In a \textit{Nash Equilibrium Problem} (NEP), for any $\nu \in
\{1, \dots, N\}$, player  $\nu$ solves the following parametrized optimization
problem, given rival decisions $\x^{-\nu}$. 
 \begin{align} \label{eqn:nash_opt}
     \min_{\x^\nu \, \in \, \Xscr^{\nu}} \, \theta_{\nu}\left(\x^{\nu},\x^{-\nu}\right) \, \triangleq \, \mathbb{E}\left[\, \tilde{\theta}_{\nu}(\x^{\nu}, \x^{-\nu},\bxi)\, \right], 
 \end{align} 
 where $\bxi: \Omega \to \Real^d$, $(\Omega, {\cal F}, \mathbb{P})$ denotes a
 probability space, $\Xi \triangleq \left\{ \xi(\omega) \, \mid \, \omega \in
 \Omega\, \right\}$, $\tilde{\theta}_{\nu}: \Real^n \times \Xi \to \Real$, and $\theta_{\nu}: \Real^n \to
 \Real$.  In the traditional model of the Nash equilibrium problem, $\Xscr^{\nu}$
 does not depend on rival decisions. Moreover, the tuple $\x^* \, \equiv \, \left\{\x^{1,*}, \cdots, \x^{N,*}\right\}$ is a Nash equilibrium if for any $\nu \, \in \, \{1, \cdots, N\}$,  
\begin{align*}\label{def:GNEQ}
    \theta_{\nu}\left( \x^{\nu,*}, \x^{-\nu,*}\right) \, \leq \, \theta_{\nu}\left( \x^{\nu}, \x^{-\nu,*}\right), \, \qquad  \forall \, \x^{\nu} \, \in \, \Xscr^{\nu}.  
\end{align*}
If each player-specific objective $\theta^{\nu}(\bullet,\x^{-\nu})$ is convex and
smooth on an open set ${\cal O}^{\nu} \supset \Xscr^{\nu}$, then $\x^*$ is a
Nash equilibrium if and only if $\x^*$ is a solution to a variational
inequality problem VI$(\Xscr, F)$, where}
$$ \Xscr \, \triangleq \, \prod_{\nu=1}^N \Xscr^{\nu} \mbox{ and } F(\x) \, \triangleq \, \pmat{  \mathbb{E}\left[ \nabla_{\x^1} \tilde{\theta}_1(\x^1,\x^{-1}, \bxi) \right] \\
            \vdots \\
         \mathbb{E}\left[ \nabla_{\x^N} \tilde{\theta}_N(\x^N,\x^{-N}, \bxi) \right]}.$$
Given an NEP $(\Theta, {\bf X})$, where $\Theta \, \triangleq \, \left\{ \theta^1, \cdots, \theta^N \, \right\}$ and ${\bf X}   \triangleq   \left(\Xscr^{{\nu}}\right)_{\nu=1}^N$, we define the bivariate function $\Psi$ as
\begin{align}\label{eqn:Psi_def}
    {\Psi} \left(\x,\y \right) \, \triangleq \, \sum_{\nu = 1}^{N} \left[\theta_{\nu}\left(\x^{\nu},\x^{-\nu}\right)-\theta_{\nu}\left(\y^{\nu},\x^{-\nu}\right) \right] \, .
\end{align} 
The function $\Psi$ is often referred to as the Nikaido-Isoda function and {was first introduced} in 1955 by Nikaido and Isoda~\cite{nikaido1955} to facilitate the analysis of NEPs.  One
notices upon inspection that the Nikaido-Isoda function {(hereafter, referred to as the {\bf NI} function)} admits an intuitive and
natural game-theoretic interpretation. That is, one may regard each summand in
the problem, 
$\theta_{\nu}\left(\x^{\nu},\x^{-\nu}\right)-\theta_{\nu}\left(\y^{\nu},\x^{-\nu}\right)$
as the improvement to player $\nu$'s payoff from making decision {$\y^{\nu}$}
\textit{instead of} {$\x^{\nu}$}, holding rival players' decisions constant at
$\x^{-\nu}$. {Consequently,} if for some $\y^{\nu}$, the quantity
$\theta_{\nu}\left(\x^{\nu},\x^{-\nu}\right)-\theta_{\nu}\left(\y^{\nu},\x^{-\nu}\right)$
is positive, it implies that player $\nu$ may improve {(i.e.{,} reduce)} their cost by changing
their strategy from $\x^{\nu}$ to $\y^{\nu}$.
Similarly, if for some $\y^{\nu}$ the quantity
$\theta_{\nu}\left(\x^{\nu},\x^{-\nu}\right)-\theta_{\nu}\left(\y^{\nu},\x^{-\nu}\right)$
is negative, it implies that player $\nu$ would increase their cost by changing
the values of their decision variables from $\x^{\nu}$ to $\y^{\nu}$. This motivates introducing a gap function $V(\bullet)$, defined as
 \begin{align}\label{eqn:V_def}
	V(\x) \triangleq \sup_{ \y \in \Xscr} \, \Psi\left(\x,\y \right), \quad \x \in \Xscr.
\end{align}
In view of our interpretation of each summand, $V(\bullet)$ represents the most that the
players can improve their payoffs in aggregate, while holding rival strategies
fixed. Importantly, $V(\bullet)$ does not represent the quantity by which the
sum of the cost functions improve if each player simultaneously takes a ``best response''
step, given perfect knowledge of the decision variables of their competitors. That is, 
\begin{align*}
V(\x) \neq \sum_{\nu = 1}^{N} \left[\theta_{\nu}\left(\x^{\nu},\x^{-\nu}\right)-\theta_{\nu}\left(\y^{\nu}{(\x)},\y^{-\nu}{(\x)}\right) \right],
\end{align*}
 where $\y^{\nu}{(\x)}$ represents the best-response of player $\nu$. 
We emphasize the game-theoretic interpretation of points $\x^*$ for which $V(\x^*) = 0$. {At such points}, no player can {unilaterally} improve {her} cost by
changing {her} decision, assuming their rivals do not change their respective
decisions. In fact, it has been shown in \cite{kanzow} that $V(\bullet)$ is
nonnegative over the set of feasible strategies and $\x^*$ is a Nash
equilibrium if and only if $\x^*$ is a feasible zero of ${V(\bullet)}$, i.e.{,} $\x^*$ is a
minimizer of the following problem   
\begin{align} 
    \min_{\x \in \Xscr} \, V(\x).
\end{align}
We revisit the function $V(\bullet)$ in the subsequent sections, as it forms the cornerstone of our analysis. 
We now proceed to provide a brief summary of prior efforts in resolving the stochastic Nash equilibrium problem.
\subsection{{Prior research on stochastic Nash equilibrium problems}} In this subsection, we review efforts made towards computing equilibria of {an} SNEP in a few different settings: (i) centralized schemes based on resolving the stochastic VI when the SNEP satisfies certain convexity requirements; (ii) (inexact) best-response schemes; (iii) (stochastic) gradient-response schemes; and {(iv)} finally techniques reliant on minimizing the {\bf NI} function in the setting of deterministic shared-constraint games. 

\medskip

\noindent (i) {\em Centralized schemes.} Recall that under convexity
requirements, the NE of a noncooperative game can be entirely captured by {the solution set of} a
variational inequality problem VI$(\Xscr,F)$. The associated map $F$ may be
monotone or non-monotone and be either single-valued or set-valued. In the
deterministic setting where $F$ is monotone and single-valued or set-valued, projection-based
techniques and proximal-point approaches may be
employed~\cite{facchinei02finite} and the resulting schemes are often equipped
with rate and complexity guarantees.  Sample-average approximation~\cite{shapiro-sa} and stochastic approximation~\cite{xu-SVI} have been employed for resolving such problems {in} both monotone settings~\cite{farzad-SVI} as well as weakenings (such as pseudomonotone)~\cite{kannan-pseudo}. 

\medskip

%When monotonicity properties fail to hold, then linesearch-based globalized techniques have been adopted for similar avenues can be adopted for computing equilibria {Add refs.}. \\

\noindent (ii) {\em Best-response schemes.} An alternative to a centralized scheme is a partially decentralized framework in which players compute a best-response, given access to the rival strategies. The resulting synchronous scheme in which players simultaneously update their best-response is referred to as a {\em best-response scheme}. Proximal variants of this, first introduced in \cite{facc2010}, require players to compute a proximal best-response, which is unique when player objectives are parametrized convex functions. Convergence of the synchronous scheme relies on the proximal best-response map being contractive~\cite{lei2020asynchronous}. An alternate approach that can be implemented in an asynchronous manner requires that the game is defined by a property of potentiality. More specifically, there exists a potential function ${\bf P}$ such that for any $\nu$, $\x^{\nu}, \y^{\nu} \in \Xscr^{\nu}$  and $\x^{-\nu} \in \Xscr^{-\nu}$, we have that 
\begin{align}
	{\bf P}(\x^{\nu},\x^{-\nu}) - {\bf P}(\y^\nu,\x^{-\nu}) = \theta_{\nu}
	(\x^{\nu},\x^{-\nu}) - \theta_{\nu}(\y^\nu,\x^{-\nu}).
\end{align}
This function allows one to monitor progress and facilitates the development of a range of schemes. Convergence guarantees rely on the potentiality property of the game.
In  \cite{lei_inexact}, an asynchronous, inexact best-response scheme was developed for the computation of Nash equilibria over a possibly time-varying network in a setting where the player-level problems were expectation valued. This result motivates our examination of a similar question, but in the setting of {\bf NI} functions instead of directly on the player problems. 

\medskip

\noindent (iii) {\em Gradient-Response
schemes.} Gradient-response schemes require
that players compute a simultaneous
gradient-projection step to update their
collective strategy tuple. Convergence
guarantees are closely tied with
monotonicity properties of the concatenated
gradient map $F$~\cite{lei_inexact}, allowing for some weakenings {such as through}
pseudomonotonicity~\cite{kannan}. Such
avenues also lend themselves to developing
distributed counterparts, as explored in the
case of deterministic and stochastic
aggregative NEPs in ~\cite{lei-distributed,
parise-distributed}, respectively. 

\medskip

%{Add some refs from \texttt{https://arxiv.org/pdf/1811.11246.pdf})}

\noindent (iv) {\em Minimization of the {\bf NI} function.} Nash equilibrium problems may be formulated as optimization problems via the {\bf NI} function. This function was examined in \cite{krawczyk} for computing the equilibrium of a noncooperative game {while optimization reformulations of a class of shared-constraint generalized NEPs} were presented and analyzed in \cite{kanzow, kanzownonsmooth}. In~ 
\cite{raghunathan}, {the authors introduce a} gradient-based {\bf NI} function and demonstrate that a gradient descent algorithm applied to this function converges to a stationary point {in unconstrained regimes}. \\

\subsection{Motivation, {contributions,} and outline}
Our focus in this paper lies in in development of avenues where neither
monotonicity of the map $F$ nor potentiality of the
game necessarily hold. We focus on the minimization of the function
$V$ by considering the equivalent problem of minimizing its regularized
counterpart $V_{\alpha}$ but in the stochastic regime. In contrast with almost all known efforts {when considering the {\bf NI} function}, our scheme can contend with stochastic regimes and provides
amongst the first known rate and complexity guarantees for such an avenue.
Given that we employ first-order schemes, it is natural that our scheme can
guarantee convergence to stationary points. However, under a suitable condition
relatively common in this thread of literature, this stationary point
proves to be a Nash equilibrium of the original problem. {Our contributions are formalized next.}
 
%As will be made apparent later, the dependence of the player strategys sets on the decisions of rivals significantly complicates the analysis of GNEPs with many basic results about existence of equilibria in the setting of NEPs do not extend directly to GNEPs. Before proceeding, we formalize the definition of a GNEP below. 
\begin{tcolorbox} {\bf Contributions.} 
In this paper we design and analyze an inexact
stochastic approximation scheme for computing stationary points of the
regularized {\bf NI} function associated with a stochastic NEP with player-specific
convex objectives, in which inexact solutions to
best-response subproblems are produced via stochastic approximation. We show
that under suitable conditions, a sequence generated by Algorithm
\ref{algorithm:NI_inexact} converges {almost surely} to a stationary point of $V\left(
\bullet\right)$. Furthermore, under suitable assumptions, we provide the following convergence rate
and complexity guarantees. 
	\begin{itemize}
		\item[(I)] For a suitably chosen fixed stepsize, the number of projection steps required to ensure that the mean of a suitably defined residual map at a randomly chosen iterate is less than $\varepsilon$ is $\mathcal{O}\left({\varepsilon^{-2}} \right)$. The associated sample complexity is $\mathcal{O}\left( {\varepsilon^{-4} }\right)$.
		\item[(II)] For a suitably chosen diminishing stepsize sequence, the analogous number of projection steps is $\mathcal{O}\left(\varepsilon^{-4} \right)$ and the sample complexity is $\mathcal{O}\left( {\varepsilon^{-6}}\right)$.
\end{itemize} 
\end{tcolorbox}

This paper is organized as follows.  In Section~\ref{sec:nep}, {we introduce the stochastic NEP, the associated {\bf NI} function as well as its regularized variant, and the relevant assumptions. Based on this foundation, we then examine the Lipschitzian properties of the function $V_{\alpha}(\bullet)$ as well as its gradient.} 
%. We present the main assumptions   we review existence and uniqueness conditions for shared-constraint games. We then introduce the regularized variant of the {\bf NI} function- {and analyze its continuity properties}.  In the setting we consider, the player specific cost functions will satisfy certain convexity assumptions and will {be} expectation valued. In this setting, it is not realistic for them to compute best response steps in finite time. It is thus natural to analyze the behavior of an inexact projected gradient based scheme when applied to the {\bf NI} function. 
In Section~\ref{sec:NI_inexact}, we introduce a sampling-enabled projected (inexact) gradient scheme, reliant on inexact subproblem resolutions. Asymptotic convergence and complexity guarantees {are presented under suitable assumptions on the sampling rate and inexactness sequence} in Section~\ref{sec:analysis}.  
%, a projected gradient descent type scheme converges almost surely to a stationary points of ${V_{\alpha}(\bullet)}$. We also characterize the rate of convergence in terms of the inexactness of the solution. {We provide a scheme by which such inexact solutions can be obtained via stochastic approximation.} 
We conclude by {recapping our main contributions  and pointing out possible
avenues for further study.}

\medskip 

\noindent {\bf Notation.} We use $\x$, $\x^{\top}$, and $\|\x\|$ to denote a column vector,
its transpose, and its Euclidean norm, respectively. We
define $f^*\triangleq \inf_{\x \in \Xscr} f(\x)$. Given a continuous mapping, i.e., $f \in C^{0}$,
we write $f \in C^{0,0}(\Xscr)$ if $f$ is Lipschitz continuous on the
set $\Xscr$ with parameter $L_0^{f}$. Given a continuously differentiable
function, i.e., $f \in C^{1}$, we write $f \in C^{1,1}(\Xscr)$ if $\nabla f$ is
Lipschitz continuous on $\Xscr$ with parameter $L_1^{f}$. We
write a.s. for ``almost surely” and $\mathbb{E}[Z]$ denotes the
expectation of a random variable $Z$ with respect to probability measure $\mathbb{P},$ unless explicitly stated otherwise. Given a scalar $u$, $[u]_+
\triangleq \max\{0,u\}$. We let $\Pi_{\Xscr}[\x]$ denote the Euclidean projection of $\x$ onto the closed convex set $\Xscr$.

%Algorithms for computing approximate stationary points of nonsmooth, nonconvex, stochastic optimization problems were the focus of chapters 1 and 2. In a game theoretical setting, players (or sometimes, agents) compete with one another by minimizing a player-specific cost function. Since in practice, players face time and resource constraints, it is natural to assume that their ability to compute solutions to their player-level problem be afflicted by inexactness. %With this in mind, a natural question arises: if players employ specific techniques, such as those introduced in ~\ref{algorithm:zo_nonconvex} and ~\ref{algorithm:quasi-newton}, under what conditions will the resulting game reach a state of equilibrium?\\

%The aim of this chapter is to address this question. Specifically, we examine what happens to a certain gap function- the {\bf NI} function- if we %equip each agent with one of these algorithms to solve their respective cost function. \\

\section{{Background and Preliminaries}}
\label{sec:nep} 
In this section we formally define the equilibrium of a stochastic NEP, briefly provide necessary and sufficient equilibrium conditions and comment on the existence {and uniqueness} of an equilibrium.
\subsection{Equilibrium Conditions}
\begin{definition}[{\bf Nash Equilibrium Problem}]
		\label{def:GNEP}
		\em
		Define the tuple of player-level cost functions and strategy sets as $\Theta$ and {${\bf X}$}, where   
		\begin{align}
            \Theta  \, \triangleq \,  \left( \theta_\nu\right)_{\nu=1}^N \, \mbox{ and }  {\bf X} \, \triangleq \, \left(\Xscr^{\nu}\right)_{\nu=1}^N,
		\end{align}
		respectively.  Then the pair $\left(\Theta, {\bf X} \right)$ {denotes} a Nash equilibrium problem. 
   {For any $\nu \in \{1, \cdots, N\}$, player $\nu$ solves the following optimization problem, parametrized by rival decisions $\x^{-\nu}$.
	$$ \min_{\x^\nu \, \in \, \Xscr^{\nu}} \, \theta_{\nu}\left(\x^{\nu},\x^{-\nu}\right) \, \triangleq \, \mathbb{E}\left[\, \tilde{\theta}_{\nu}(\x^{\nu}, \x^{-\nu},\bxi)\, \right], $$
	where $\bxi: \Omega \to \Real^d$, $(\Omega, {\cal F}, \mathbb{P})$ denotes a
	probability space, $\Xi \triangleq \left\{ \xi(\omega) \, \mid \, \omega \in
	\Omega\, \right\}$, $\tilde{\theta}_{\nu}: \Real^n \times \Xi \to \Real$, and $\theta_{\nu}: \Real^n \to
	\Real$.}
$\hfill \Box$
\end{definition}

There has been an effort to weaken convexity assumptions in player
specific objectives $\theta_{\nu}(\bullet,\x^{-\nu})$ for any $\x^{-\nu} \, \in
\, \Xscr^{-\nu}$ and $\nu \, \in \, \left\{\, 1, \cdots, N\, \right\}$, existence
guarantees for general nonconvex settings remain elusive (absent additional
structure). In this work, we focus on regimes where convexity of player-specific objectives holds, given rival decisions. This is captured in our main assumption.

%\begin{definition}[{\bf Convexity Assumption}] \label{ass:convexity-assumption}
%    \noindent For every player $\nu$, the player-level cost function $\theta_{\nu}$ is convex and the set $\Xscr_{\nu}(\x^{-\nu})$ is closed and convex. 
%\end{definition}
\begin{assumption}[{\bf Properties {of player problems}}] \label{ass:ass-1} \em {For any $\nu \, \in \, \{1, \cdots, N\}$, the following hold.}

    \noindent {\bf (a)} For any $\x^{-\nu} \, \in \, \Xscr^{-\nu}$, player $\nu$'s cost function $\theta_{\nu}\left(\bullet,\x^{-\nu}\right)$ is convex on $\Xscr^{\nu}$. 

\noindent {\bf (b)} {$\theta_{\nu}$ is $L_0^{\nu}$-Lipschitz and $L_1^{\nu}$-smooth on $\Xscr$.}

\noindent {\bf (c)} The set $\Xscr^{\nu}$ is compact and convex.  %{\bf {ToDo:}} {Do we need the compactness assumption? Wouldn't closed convex sets suffice?} {Need boundedness for Lipschitzian claims.}
$\hfill \Box$
\end{assumption}

\medskip

Consider the NEP given by $(\Theta,{\bf X})$ under Assumption~\ref{ass:ass-1}. {Consequently, the necessary and sufficient conditions of optimality of player $\nu$'s problem, given rival decisions $\x^{-\nu}$, are compactly captured} by
the variational inequality problem VI$(\Xscr^{\nu},
\nabla_{\x^\nu}\theta_{\nu}(\bullet,\x^{-\nu}))$, i.e.{,} $\x^{*,\nu}$ satisfies 
\begin{align}\label{VI-nu}
	 (\y^{\nu} - {\x^{*,\nu}})^\top \nabla_{\x^{\nu}} \theta_{\nu}(\x^{*,\nu},\x^{*,-\nu}) \, \geq \, 0, \qquad \forall \, \y^{\nu} \, \in \, \Xscr^{\nu}.
  \end{align}
%\lm{ToDO: Maybe we add a * to the below expression, or remove it from the above to maintain consistency?}
    Consequently, $\x \, \equiv \, (\x^{\nu})_{\nu=1}^N$ is a {Nash Equilibrium of $(\Theta,{\bf X})$} if and only if $\x^{\nu}$ {solves} VI$(\Xscr^{\nu}, \nabla_{\x^\nu}\theta_{\nu}(\bullet,\x^{-\nu}))$ for $\nu \, \in \, \{1, \cdots, N\}$, i.e.{,} 
\begin{align}\label{VI-all}
    \begin{aligned}
        (\y^{1} - \x^{1})^\top \nabla_{\x^{1}} \theta_{1}({\x^{1}},\x^{-1}) \, & \geq \, 0, \qquad \forall \, \y^1 \, \in \, \Xscr^{1} \\
        & \vdots \\
        (\y^{N} - \x^{N})^\top \nabla_{\x^{N}} \theta_{N}({\x^{N}},\x^{-N}) \, & \geq \, 0, \qquad \forall \, \y^{N} \, \in \, \Xscr^{N}. 
    \end{aligned}
\end{align}
By appealing to a result from~\cite[Ch.~1]{facchinei02finite}, it can be seen that the collection of $N$ coupled variational inequality problems is equivalent to a single variational inequality problem VI$(\Xscr, F)$ where 
$$ \Xscr \, \triangleq \, \prod_{\nu=1}^N \Xscr^{\nu} \, \mbox{ and } F(\x) \, \triangleq \pmat{\nabla_{\x^{1}} \theta_{1}({\x^{1}},\x^{-1}) \\
\vdots \\\nabla_{{\x^{N}}} \theta_{N}({\x^{N}},\x^{-N})},$$
respectively. More specifically, $\x^*$ is a Nash equilibrium if and only if 
\begin{align}
    (\y - \x^*)^\top F(\x^*)\, \geq \, 0, \qquad \forall \, \y \, \in \, \Xscr.
\end{align}
Consequently, under convexity requirements, the existence of a Nash Equilibrium
{can be reduced to checking the} solvability of a suitably defined
variational inequality problem.  Next, we present some basic results on the
existence and uniqueness of equilibria for NEPs. Our first result
considers the Nash equilibrium problem under suitable convexity
requirements~\cite{facchinei02finite}.  

\begin{theorem}[{\bf Existence of NE}] \em Consider the NEP$(\Theta,{\bf X})$. Suppose Assumption~\ref{ass:ass-1} holds. Then a Nash equilibrium exists if one of the following hold{s}. 

    \noindent (a) $\Xscr^{\nu}$ is bounded for {every} $\nu \in \{1, \dots, N\}.$
    
    \noindent (b) There exists a vector $\x^{\rm ref} \, \in \, \Xscr$ such that 
    $$ \liminf_{ \x \, \in \, \Xscr, \|\x\| \to \infty} \, F(\x)^\top(\x-\x^{\rm ref}) \, > \, 0.  
    \hspace{2in} \Box $$
\end{theorem}

Similarly, uniqueness claims on the NEP require assessing the properties of the map $F$ associated with the related VI$(\Xscr,F)$. The following result provides a formal set of conditions for the  uniqueness of a Nash Equilibrium by leveraging a condition that ensures that VI$(\Xscr, F)$ admits a unique solution~\cite{facchinei02finite}. 

%\gap

\begin{theorem}[{\bf Uniqueness of NE}] \em  
Consider the NEP$(\Theta,{\bf X})$. Suppose
Assumption~\ref{ass:ass-1} holds and the map $F$ is strongly monotone over $\Xscr$.  Then this game admits a unique
Nash equilibrium.  $\hfill \Box$ 
\end{theorem}

%Importantly, player-convexity, and convexity of $\theta_{\nu}$ are not the same. We will refer to and invoke this assumption by name, frequently, in the following sections. 
\subsection{The Regularized {\bf NI} Function}
We now introduce  the regularized counterparts of $\Psi(\bullet,\bullet)$ and $V(\bullet)$ introduced earlier {via \eqref{eqn:Psi_def} and \eqref{eqn:V_def}, respectively}. Such a regularization has been considered and analyzed in this setting~\cite{kanzow} {as well as in related contexts}~\cite{fukushima1992,gurkan,Mastroeni2003}. As shown in ~\cite{kanzow}, regularizing the {\bf NI} function yields a variant of the value function that is continuously differentiable. The benefit of doing so will be made readily apparent when we examine the maximizer of {$\Psi$}. Given a scalar $\alpha>0$, define first
\begin{align}\label{eqn:Psi_alpha}
	\Psi_{\alpha} \left(\x,\y \right) \, \triangleq \, \sum_{\nu = 1}^{N} \left[\theta_{\nu}\left(\x^{\nu},\x^{-\nu}\right)-\theta_{\nu}\left(\y^{\nu},\x^{-\nu}\right)- \frac{\alpha}{2} \|\x^{\nu} - \y^{\nu} \|^2 \right] \, .
\end{align}
Note that {$\Psi_{\alpha} \left(\x,\y \right) = \Psi(\x,\y) - \frac{\alpha}{2}\|\x - \y\|^2$}. Similarly, for any $\x \, \in \, \Xscr$, 
\begin{align}\label{def:valpha}
	V_{\alpha}(\x) & \triangleq \sup_{\y \in \Xscr} \, \Psi_{\alpha}\left(\x,\y \right) \, = \, \sum_{\nu = 1}^{N} \left[ \theta_{\nu}\left(\x^{\nu},\x^{-\nu}\right)- \min_{\y^\nu \in {\Xscr}^{\nu}}\left[ \theta_{\nu}\left(\y^{\nu},\x^{-\nu}\right)+ \frac{\alpha}{2} \|\x^{\nu} - \y^{\nu} \|^2 \right] \right], 
\end{align}
where the replacement of the supremum by the minimum is justified by the continuity of $V_{\alpha}(\bullet),$ and the {compactness} of $\Xscr$. 
We state the following result from ~\cite{kanzow} without proof. 
\begin{proposition}[{\bf Properties of the regularized {\bf NI} function $V_{\alpha}{(\bullet)}$}]\label{thm:v_alpha_properties} \em Let {Assumption}~\ref{ass:ass-1} hold. Suppose $V_{\alpha}(\bullet)$ is defined in \eqref{def:valpha}. Then the following hold.

	\noindent(a) $V_{\alpha} \left( \x\right) \geq 0$ for all $\x \in \Xscr$. 

	\noindent(b) $\x^*$ is a Nash equilibrium if and only if ${\x^*} \in \Xscr$ and $V_{\alpha}(\x^*) = 0 \,$ . 

    \noindent(c) For every $ \x  \in X$, there exists a unique vector $\y_{\alpha}(\x) \equiv \left( \y_{\alpha}^{\nu}(\x)\right)_{\nu=1}^N$ such that for every $\nu = 1, \dots , N$,  
	\begin{align} \label{prob-subprob}
        \y_{\alpha}^{\nu}(\x)  \, \triangleq \, \argmin_{{\y^\nu \in {\Xscr}^{\nu}}} \left[ \theta_{\nu}\left(\y^{\nu},\x^{-\nu}\right)+ \frac{\alpha}{2} \|\x^{\nu} - \y^{\nu} \|^2 \right]. \hspace{1in} \Box
	\end{align}
\end{proposition}
%\gap

We now consider the question of whether or not $V_{\alpha}(\bullet)$ is smooth, provided
that the player-level cost functions satisfy certain smoothness conditions.
{Continuous differentiability of $V_{\alpha}$} was proven in ~\cite{kanzow}; we repeat the proof here as it
plays an important role in the analysis that follows. {In addition, we derive the Lipschitz continuity and $L$-smoothness of $V_{\alpha}(\bullet)$.} Before proceeding, we
state a well-known result that will help us in establishing the smoothness of
$V_{\alpha}(\bullet)$. 

\begin{lemma}[{\bf Danskin's Theorem}]\label{lemma:danskin-two}
\em
    {Suppose $\mathcal{Y}$ is a compact set and  $\phi(\x,\y) : \Xscr \times \mathcal{Y} \to \mathbb{R}$ is a continuous function such that $\nabla_{\x}\phi(\x,\y)$ exists and is continuous on $\Xscr \times \mathcal{Y}$.}  Then the following hold for the {value} function $U$, defined as 
    $U(\x) \, \triangleq \, {\displaystyle \max_{\y \in \mathcal{Y}}} \, \phi(\x,\y).$ 

    \noindent (i)  Suppose $Y(\x)$ is defined as
    \begin{align}
        Y(\x) \, \triangleq \, \left\{ \, \bar{y} \, \mid \, \phi(\x,\bar{y}) \, = \, \max_{\y \in \mathcal{Y}} \, \phi(\x,\y) \, \right\}. 
    \end{align}
    Then $U$ is directionally differentiable at $\x$, and $U'(\x ; d) = {\displaystyle \max_{\y \, \in \,  Y(\x)}}\phi (\x ,\y;d).$  

	\noindent(ii) If $Y(\x)$ is a singleton, $\nabla_{\x}U(\x) = \nabla_{\x}\phi(\x,{\y}) \Bigr \rvert_{\y = \y(\x)},$ where $\y(\x)$ uniquely maximizes $\phi(\x,\y)$ for {$\y \in \mathcal{Y}$}.  $\hfill \Box$
\end{lemma}

%\gap
{We begin by proving that $\y^{\nu}_{\alpha}(\bullet)$ is a Lipschitz continuous map. Our result leverages Lipschitzian claims on solution maps  of strongly monotone variational inequality problems provided by Dafermos~\cite{dafermos88sensitivity}.}
{\begin{lemma}\label{lem:lips_cont_y_alpha} \em 
Suppose Assumption~\ref{ass:ass-1} holds. Let $\nu \, \in \, \left\{\, 1, \cdots, N\, \right\}$ be given. 
Suppose $\alpha$ is an arbitrary positive scalar such that $\alpha >L_G^{\nu}$. Then   
$$\| \y_{\alpha}^{\nu}(\x_1) - \y_{\alpha}^{\nu}(\x_2)\| \leq 
    {L_{0}^{\y_{\alpha}}} \left\| \x_1- \x_2  \right\| {\mbox{ for any $\x_1, \x_2 \, \in \, \Xscr$,}} $$  
    where ${L_{0}^{\y_{\alpha}}} \triangleq \left( \frac{   \alpha+L_G^{\nu} }{  \alpha - L_G^{\nu}} \right) $. 
    
    %\fyy{FY: Some minor changes; First, I removed $0 < (1- \gamma (\alpha - L_G^{\nu})) < 1$ because it is implied from $\gamma,\alpha>0$, $0 < (1-\gamma \alpha) < 1$, and $\alpha > L_g^{\nu}$ and seems unnecessary to be added as an additional condition. Second, I further noticed that $\gamma$ may not be needed to be mentioned in the statement of the lemma, as it is a parameter that is only used as a scalar to derive the proof. So, I dropped $\gamma$ from the statement of the lemma.}
    
\end{lemma}}
\begin{proof}
	In view of \eqref{prob-subprob}, $\y_{\alpha}^{\nu}(\x) =   \Pi_{\Xscr^{\nu}} \left[ \y_{\alpha}^{\nu}(\x)- \gamma (\nabla_{\x^{\nu}} \theta_{\nu}(\y_{\alpha}^{\nu}(\x),\x^{-\nu}){+ \alpha (\y_{\alpha}^{\nu}(\x) - \x^{\nu})})\right]$ for any $\gamma>0$. Using this stationarity condition, invoking the  non-expansivity of the Euclidean projector, and choosing $\gamma$ such that $0 < \gamma \alpha < 1$ (e.g., $\gamma:=\frac{1}{2\alpha}$), we obtain
		\begin{align*}
 \notag	\| \y_{\alpha}^{\nu}(\x_1) - \y_{\alpha}^{\nu}(\x_2)\| & = \left\| \Pi_{\Xscr^{\nu}} \left[ \y_{\alpha}^{\nu}(\x_1) - \gamma \bigl( \nabla_{\x^{\nu}} \theta_{\nu}(\y_{\alpha}^{\nu}(\x_1),\x_1^{-\nu}) + \alpha (\y_{\alpha}^{\nu}(\x_1) - \x_1^{\nu}) \bigr) \right]  \right. \\
	\notag& \left. - \Pi_{\Xscr^{\nu}} \left[ \y_{\alpha}^{\nu}(\x_2) - \gamma \bigl( \nabla_{\x^{\nu}} \theta_{\nu}(\y_{\alpha}^{\nu}(\x_{2}),\x_2^{-\nu}) + \alpha (\y_{\alpha}^{\nu}(\x_2) - \x_2^{\nu}) \bigr) \right] \right\| \\
	& \le \left\|  \left[ \y_{\alpha}^{\nu}(\x_1) - \gamma \bigl( \nabla_{\x^{\nu}} \theta_{\nu}(\y_{\alpha}^{\nu}(\x_1),\x_1^{-\nu}) + \alpha (\y_{\alpha}^{\nu}(\x_1) - \x_1^{\nu}) \bigr) \right] \right. \\ 
	\notag & \left. -  \left[ \y_{\alpha}^{\nu}(\x_2) - \gamma \bigl( \nabla_{\x^{\nu}} \theta_{\nu}(\y_{\alpha}^{\nu}(\x_{2}),\x_2^{-\nu}) + \alpha (\y_{\alpha}^{\nu}(\x_2) - \x_2^{\nu}) \bigr)  \right] \right\| \\
	& \le (1-\gamma \alpha) \|\y_{\alpha}^{\nu}(\x_1) - \y_{\alpha}^{\nu}(\x_2)\| + \gamma   
\left\|  \nabla_{\x^{\nu}} \theta_{\nu}(\y_{\alpha}^{\nu}(\x_1),\x_1^{-\nu}) -  \nabla_{\x^{\nu}} \theta_{\nu}(\y_{\alpha}^{\nu}(\x_{2}),\x_2^{-\nu}) \right\| \\
	& + \gamma \alpha \| \x_1^{\nu} - \x_2^{\nu}\|\\
& \le    (1-\gamma \alpha) \|\y_{\alpha}^{\nu}(\x_1) - \y_{\alpha}^{\nu}(\x_2)\| + \gamma L_G^{\nu}  
\left\| \pmat{ \y_{\alpha}^{\nu}(\x_1) - \y_{\alpha}^{\nu}(\x_2) \\
	\x_1^{-\nu} - \x_2^{-\nu}} \right\| + \gamma \alpha \| \x_1^{\nu} - \x_2^{\nu}\|.
\end{align*}
We may derive a bound on the penultimate term on the RHS as follows.
\begin{align*}
\left\| \pmat{ \y_{\alpha}^{\nu}(\x_1) - \y_{\alpha}^{\nu}(\x_2) \\
	\x_1^{-\nu} - \x_2^{-\nu}} \right\| & \le \sqrt{ \sum_{i=1}^n [\y_{\alpha}^{\nu}(\x_1) - \y_{\alpha}^{\nu}(\x_2)]_i^2 + \sum_{i=1}^n  [\x_1 - \x_2]_i^2 } \\
	 & \le    \sqrt{ \left\|\y_{\alpha}^{\nu}(\x_1) - \y_{\alpha}^{\nu}(\x_2)\right\|^2} + \sqrt{\left\|\x_1 - \x_2 \right\|^2} 
	 = \left\|\y_{\alpha}^{\nu}(\x_1) - \y_{\alpha}^{\nu}(\x_2)\right\| + \left\|\x_1 - \x_2 \right\|, 
\end{align*}
{where $\sqrt{a+b} \le \sqrt{a}+\sqrt{b}$ for $a,b \ge 0$.}
Consequently, we have that 
		\begin{align*}
 \notag	\| \y_{\alpha}^{\nu}(\x_1) - \y_{\alpha}^{\nu}(\x_2)\| &  \le    (1-\gamma \alpha) \|\y_{\alpha}^{\nu}(\x_1) - \y_{\alpha}^{\nu}(\x_2)\| + \gamma  L_G^{\nu} \left(\left\|\y_{\alpha}^{\nu}(\x_1) - \y_{\alpha}^{\nu}(\x_2)\right\| + \left\|\x_1 - \x_2 \right\| \right) \\
&  + \gamma \alpha \| \x_1^{\nu} - \x_2^{\nu}\|.
\end{align*}
This implies that
		\begin{align*}
 \notag	\gamma(\alpha-L_G^{\nu}) \| \y_{\alpha}^{\nu}(\x_1) - \y_{\alpha}^{\nu}(\x_2)\| &  \le   \gamma(\alpha+L_G^{\nu}) \| \x_1^{\nu} - \x_2^{\nu}\|.
\end{align*}
Recall that $\gamma>0$ and $\alpha > L_G^{\nu}$. Then we have that 
\begin{align*}
 \notag	\| \y_{\alpha}^{\nu}(\x_1) - \y_{\alpha}^{\nu}(\x_2)\| &  \le     \left(\frac{\gamma  (\alpha+L_G^{\nu}) }{\gamma (\alpha - L_G^{\nu})} \right) \left\|\x_1 - \x_2 \right\| \, = \, \left(\frac{  (\alpha+L_G^{\nu}) }{  (\alpha - L_G^{\nu})} \right) \left\|\x_1 - \x_2 \right\|.
\end{align*}
\end{proof}
%\fyy{FY. I commented out the remark about $\gamma$ and $\alpha$, following my earlier comment about Lemma 2.7.}
%{{\bf Remark.} \fyy{To find} $\gamma, \alpha \fyy{>0}$ such that $ (1-\gamma \alpha) \in (0,1)$ \fyy{and}  $\alpha > L_G^{\nu} $, \fyy{we may choose, for example,} $\gamma = 1/(2\alpha)$ and $\alpha = 2L_G^{\nu}$. Then we have $\alpha > L_G^{\nu}$ \fyy{and}  $
%	(1-\gamma \alpha)  = \tfrac{1}{2}$. 
%$\hfill$ $\Box$}

With Danskin's theorem in hand, we are now ready to state a few important propositions characterizing the smoothness of  $V_{\alpha}{(\bullet)}$. {Part (i) of the result below has been proven in ~\cite{kanzow} while parts (ii) and (iii) have not been addressed in prior literature.} 
\begin{proposition}[Smoothness {and Lipschitzian} properties of $V_{\alpha}{(\bullet)}$]\label{prop:V_alpha} \em
    \noindent  Suppose Assumption~\ref{ass:ass-1} holds.  Then the following hold.
    
    \noindent  {\bf(i)} $\nabla_{{\x}} V_{\alpha}(\x)$ is given by 
$$\nabla_{{\x}} V_{\alpha}(\x) = \nabla_{{\x}} {\Psi_{\alpha}}(\x,\y) \big|_{\y = \y_\alpha(\x)},$$ where 
\begin{align*}
    \nabla_{\x} V_{\alpha} \left( \x \right) &= \sum_{\nu =1}^{N} \left[\nabla_{{\x}} \theta_{\nu}\left(\x^{\nu},\x^{-\nu}\right)- \nabla_{{\x}} \theta_{\nu}\left({\y_{\alpha}^{\nu}(\x)},\x^{-\nu}\right)\right] + \begin{bmatrix}
        \nabla_{\x^1}\theta_{1}\left({\y_{\alpha}^{1}(\x)},\x^{-1}\right) \\
		\vdots \\
        \nabla_{\x^N}\theta_{N}\left({\y_{\alpha}^{N}(\x)},\x^{-N}\right)
    \end{bmatrix} - \alpha \left(\x- \y_{\alpha}(\x) \right).
\end{align*}
    \noindent {\bf(ii)} {The function $V_{\alpha}(\bullet)$ is $L_0^{V_{\alpha}}$-Lipschitz on $\Xscr$}, i.e., {for any $\bu , \bv \in \Xscr$, $V_{\alpha}$ satisfies 
\begin{align*}
	\| V_{\alpha}(\bu) - V_{\alpha}(\bv) \| \leq L_{0}^{V_{\alpha}} \|\bu - \bv \|, 
\end{align*}
    where $L_{0}^{V_{\alpha}} \, \triangleq \, {\displaystyle \sum_{\nu=1}^N} \left(L_{0}^{\nu}(1+\sqrt{1+{(L_{0}^{\y_{\alpha}})}^{2}}) + {{4}\alpha C^{\nu}} (1+L_{0}^{\y_{\alpha}})\right),$}
where $C^{\nu}$ is the diameter of the set $\Xscr^{\nu}$. \\
    \noindent {\bf(iii)} {The function $V_{\alpha}(\bullet)$ is $L_1^{V_{\alpha}}$-smooth on $\Xscr$}, i.e., for any $\bu , \bv \in \Xscr$, $V_{\alpha}$ satisfies 
    {\begin{align}\label{eqn:V_smooth}
    \left\| \nabla_{{\x}} V_{\alpha}(\bu) - \nabla_{{\x}} V_{\alpha}(\bv)\right\| & \leq {L_1^{V_{\alpha}}} \left\|\bu - \bv   \right\|, 
\end{align}
where $L_1^{V_{\alpha}} \, \triangleq \, 
{\left({2}\sum_{\nu =1}^{N} {\left( L_{1}^{\nu}  + L_{1}^{\nu}  \sqrt{1+ {(L_{0}^{\y_{\alpha}})^2}}\right) }+ \alpha \left( 1 + L_{0}^{\y_{\alpha}}\right) \right)}
    .$
}
%\begin{align*}
%	\left\| \nabla V_{\alpha}(\bu) - \nabla V_{\alpha}(\bv)\right\| \leq N L_{\textbf{max}} (2 + 3 \max_{\nu \, \in \{1, \cdots, N\}}\left(\tfrac{L_{\nu}(1+a^{\nu})\alpha^2}{1- (1+\tfrac{1}{a^{\nu}}) (1-2\gamma_{k} \alpha + (\gamma_{k} L_{1}^{\nu})^2)}\right)  \|\bu - \bv\| \,
%\end{align*}
%where $\y_{\alpha}(\x)$ is as defined in \ref{thm:v_alpha_properties}. 
\end{proposition}
\begin{proof}
    \noindent (i) Note that the mapping $\Psi_{\alpha}\left(\x, \y \right) = \Psi \left(\x, \y \right) - \frac{\alpha}{2} \|\x- \y \|^2 $ is strongly concave in $\y$ for fixed $\x$. Consequently, the set of solutions of ${\displaystyle \sup_{\y \in \Xscr}} \, \Psi_{\alpha}\left(\x,\y \right)$ is a singleton. Applying Danskin's theorem (cf. Lemma~\ref{lemma:danskin-two}), from \eqref{def:valpha}, it follows that $V_{\alpha}$ is differentiable {and its} gradient {is}  given by 
        $$\nabla_\x V_{\alpha}{(\x)} = \nabla_{\x} \Psi_{\alpha} \left(\x ,\y \right) \Bigr \rvert_{\y = \y_{\alpha}(\x)}.$$ 
    Let us consider \eqref{eqn:Psi_alpha}. Calculating $\nabla_{\x} \Psi_{\alpha} \left(\x ,\y\right)$  by proceeding term by term, it is easy to see that the gradient of the first and third terms on the RHS in \eqref{eqn:Psi_alpha} are given by
    \begin{align*}
    	\nabla_{\x} \sum_{\nu =1}^{N} \theta_{\nu}\left(\x^{\nu},\x^{-\nu}\right) = \sum_{\nu =1}^{N} \nabla_{\x} \theta_{\nu}\left(\x^{\nu},\x^{-\nu}\right) \,
	\mbox{ and } \nabla_{\x} \left(- \frac{\alpha}{2} \|\x- \y \|^2\right)  = - \alpha \left(\x- \y \right).
	\end{align*}
	We now examine the term $ -\sum_{\nu =1}^{N} \theta_{\nu}\left({\y}^{\nu},\x^{-\nu}\right).$ Observe that the $\nu$th components of $\nabla \theta_{\nu}\left({\y}^{\nu},\x^{-\nu}\right)$ are zero, {i.e., $\nabla_{\x^{\nu}} \theta_{\nu}\left({\y}^{\nu},\x^{-\nu}\right) = 0$}, where $y^{\nu}$ is a fixed parameter. 
%\fyy{FY: It is unclear why $\nabla_{\x^{\nu}} \theta_{\nu}\left({\y}^{\nu},\x^{-\nu}\right) = 0$. I might be overlooking something, but my understanding is that before we substitute $\y$ by $\y_{\alpha}(\x)$ in applying Danskin's theorem, vector $\y$ is an arbitrary vector.} 
We may then represent the gradient of the term of interest by  
	\begin{align*}
		{-}\nabla_{{\x}} \sum_{\nu =1}^{N} \theta_{\nu}\left(\y^{\nu},\x^{-\nu}\right) = { {-}\sum_{\nu =1}^{N}  \nabla_{{\x}} \theta_{\nu}\left(\y^{\nu},\x^{-\nu}\right)}  + {\underbrace{\begin{bmatrix}
			\nabla_{\x^1}\theta_{1}\left({\y^{1}},\x^{-1}\right) \\
			\vdots \\
			\nabla_{\x^N}\theta_{N}\left({\y^{N}},\x^{-N}\right)
		\end{bmatrix}}_{{\, = \, 0}}},
	\end{align*}
where this new vector with component gradients acts to {``}cancel out{"} the requisite nonzero components of the preceding term. Putting these three terms together,
    \begin{align*}
        \nabla_{\x} \Psi_{\alpha} \left(\x ,{\y} \right) = \sum_{\nu =1}^{N} \left[\nabla_{\x} \theta_{\nu}\left(\x^{\nu},\x^{-\nu}\right)- \nabla_{\x} \theta_{\nu}\left(\y^{\nu},\x^{-\nu}\right)\right] + \begin{bmatrix}
            \nabla_{\x^1}\theta_{1}\left({\y^{1}},\x^{-1}\right) \\
		\vdots \\
            \nabla_{\x^N}\theta_{N}\left({\y^{N}},\x^{-N}\right)
	\end{bmatrix} - \alpha \left(\x- \y \right) \, ,
    \end{align*}
    and  evaluating at ${\y=}\y_{\alpha}(\x)$ yields the desired result. 
    
    \noindent {\bf(ii)} We can calculate the Lipschitz constant of $V_{\alpha}(\bullet)$  by using the definition of $V_{\alpha}(\bullet)$ and the Lipschitz continuity of $\y_{\alpha}(\bullet)$, as shown by the next inequalities for any $\bu, \bv \in \Xscr$.
\begin{align*}
    \| V_{\alpha}(\bu) - V_{\alpha}(\bv) \|&  \leq \left\| \sum_{\nu=1}^N \left[ \theta_{\nu}\left( \bu^{\nu}, \bu^{-\nu} \right) - \theta_{\nu}\left( {\y_{\alpha}^{\nu}}(\bu), \bu^{-\nu} \right) + \frac{\alpha}{2} \|\bu^{\nu} -\y_{\alpha}^{\nu}(\bu) \|^2 \right]  \right. \\
    &  -\left.\sum_{\nu=1}^N \left[ \theta_{\nu}\left( \bv^{\nu}, \bv^{-\nu} \right) - \theta_{\nu}\left( {\y_{\alpha}^{\nu}}(\bv), \bv^{-\nu} \right) + \frac{\alpha}{2} \| \bv^{\nu} -\y_{\alpha}^{\nu}(\bv) \|^2 \right] \right\| \\
    & \leq {\sum_{\nu=1}^N L_{0}^{\nu}\|\bu -\bv\| + \sum_{\nu=1}^N L_{0}^{\nu} \left\| \pmat{ \y_{\alpha}^{\nu}(\bu) - \y_{\alpha}^{\nu}(\bv) \\
    \bu^{-\nu} - \bv^{-\nu}} \right\|} \\
    & + {\frac{\alpha}{2}}{\sum_{\nu=1}^N \|\bu^{\nu} - \bv^{\nu} + \y^{\nu}_{\alpha}(\bv) - \y^{\nu}_{\alpha}(\bu)\|
    \|\bu^{\nu} + \bv^{\nu} {-\y^{\nu}_{\alpha}(\bu)} - \y^{\nu}_{\alpha}(\bv)\|}\\
    & \leq {\sum_{\nu=1}^N L_{0}^{\nu}\|\bu -\bv\| + \sum_{\nu=1}^N L_{0}^{\nu}\sqrt{\left( \| \y_{\alpha}^{\nu}(\bu) - \y_{\alpha}^{\nu}(\bv)\|^2 + \|\bu - \bv\|^2\right)}} \\
    & + {\sum_{\nu=1}^N {2\alpha C^{\nu}} \left(\|\bu^{\nu} - \bv^{\nu}\| + \|\y^{\nu}_{\alpha}(\bv) - \y^{\nu}_{\alpha}(\bu)\|\right)}\\
    & \leq {\sum_{\nu=1}^N L_{0}^{\nu}\|\bu -\bv\| + \sum_{\nu=1}^N L_{0}^{\nu}\sqrt{1+{(L_{0}^{\y_{\alpha}})}^2}\|\bu - \bv\|} 
     + {\sum_{\nu=1}^N {{4}\alpha C^{\nu}} (1+L_{0}^{\y_{\alpha}}) \|\bu^{\nu} - \bv^{\nu}\|}\\
    & \leq {\sum_{\nu=1}^N \left(L_{0}^{\nu}(1+\sqrt{1+{(L_{0}^{\y_{\alpha}})}^2}) + {2\alpha C^{\nu}} (1+L_{0}^{\y_{\alpha}})\right) \left\| \bu - \bv \right\|},	
\end{align*}
where the second inequality relies on Assumption~\ref{ass:ass-1} and $\|{\bf p}\|^2 - \|{\bf q}\|^2 = ({\bf p}-{\bf q})^\top({\bf p}+{\bf q}) \leq \|{\bf p}-{\bf q}\|\|{\bf p}+{\bf q}\|$ for all ${\bf p},{\bf q} \in \mathbb{R}^{n_\nu}$ and  in the third inequality we use $ \|\bu^{\nu} + \bv^{\nu} {-\y^{\nu}_{\alpha}(\bu)} - \y^{\nu}_{\alpha}(\bv)\| \leq 4C^{\nu}$ for some $C^{\nu} > 0$, in view of boundedness of $\Xscr^{\nu}$. {Further, the fourth inequality follows from Lemma~\ref{lem:lips_cont_y_alpha}.}

\noindent {\bf(iii)}  We now show that $\nabla_x V_{\alpha}{(\bullet)}$ is Lipschitz continuous {on $\Xscr$}. From part {\bf(i)},
\begin{align*}
    & \quad \left\| \nabla_{{\x}} V_{\alpha}(\bu) - \nabla_{{\x}} V_{\alpha}(\bv)\right\|\\
    & = \left\|  \sum_{\nu =1}^{N} \left[\nabla_{{\x}} \theta_{\nu}\left(\bu^{\nu},\bu^{-\nu}\right)- \nabla_{{\x}} \theta_{\nu}\left(\y_{\alpha}^{\nu}(\bu),\bu^{-\nu}\right)\right] + \begin{bmatrix}
		\nabla_{\x^1}\theta_{1}\left(\y_{\alpha}^{1}(\bu),\bu^{-1}\right) \\
		\vdots \\
		\nabla_{\x^N}\theta_{N}\left(\y_{\alpha}^{N}(\bu),\bu^{-N}\right)
	\end{bmatrix} - \alpha \left(\bu- \y_{\alpha}(\bu) \right) \right.  \\
&   \left. - \sum_{\nu =1}^{N} \left[\nabla_{{\x}} \theta_{\nu}\left(\bv^{\nu},\bv^{-\nu}\right)- \nabla_{{\x}} \theta_{\nu}\left(\y_{\alpha}^{\nu}(\bv),\bv^{-\nu}\right)\right] - \begin{bmatrix}
	\nabla_{\x^1}\theta_{1}\left(\y_{\alpha}^{1}(\bv),\bv^{-1}\right) \\
	\vdots \\
	\nabla_{\x^N}\theta_{N}\left(\y_{\alpha}^{N}(\bv),\bv^{-N}\right)
\end{bmatrix} + \alpha \left(\bv- \y_{\alpha}(\bv) \right) \right\| \\
& = \left\|  \sum_{\nu =1}^{N} \left( \left[\nabla_{{\x}} \theta_{\nu}\left(\bu^{\nu},\bu^{-\nu}\right)- \nabla_{{\x}} \theta_{\nu}\left(\y_{\alpha}^{\nu}(\bu),\bu^{-\nu}\right)\right] - \left[\nabla_{{\x}} \theta_{\nu}\left(\bv^{\nu},\bv^{-\nu}\right)- \nabla_{{\x}} \theta_{\nu}\left(\y_{\alpha}^{\nu}(\bv),\bv^{-\nu}\right)\right] \right) \right. \\
& + \left. \begin{bmatrix}
	\nabla_{\x^1}\theta_{1}\left(\y_{\alpha}^{1}(\bu),\bu^{-1}\right) -\nabla_{\x^1}\theta_{1}\left(\y_{\alpha}^{1}(\bv),\bv^{-1}\right) \\
	\vdots \\
	\nabla_{\x^N}\theta_{N}\left(\y_{\alpha}^{N}(\bu),\bu^{-N}\right) -\nabla_{\x^N}\theta_{N}\left(\y_{\alpha}^{N}(\bv),\bv^{-N}\right)
\end{bmatrix} - \alpha \left(\bu-\bv \right) + \alpha \left( \y_{\alpha}(\bu) - \y_{\alpha}(\bv)\right)   \right\|  \\ 
& \leq \left\|  \sum_{\nu =1}^{N} \left( \left[\nabla_{{\x}} \theta_{\nu}\left(\bu^{\nu},\bu^{-\nu}\right)- \nabla_{{\x}} \theta_{\nu}\left(\y_{\alpha}^{\nu}(\bu),\bu^{-\nu}\right)\right] - \left[\nabla_{{\x}} \theta_{\nu}\left(\bv^{\nu},\bv^{-\nu}\right)- \nabla_{{\x}} \theta_{\nu}\left(\y_{\alpha}^{\nu}(\bv),\bv^{-\nu}\right)\right] \right) \right\| 
\end{align*}
\begin{align*}
& + \left\| \begin{bmatrix}
	\nabla_{{\x^1}}\theta_{1}\left(\y_{\alpha}^{1}(\bu),\bu^{-1}\right) -\nabla_{{\x^1}}\theta_{1}\left(\y_{\alpha}^{1}(\bv),\bv^{-1}\right) \\
	\vdots \\
	\nabla_{{\x^N}}\theta_{N}\left(\y_{\alpha}^{N}(\bu),\bu^{-N}\right) -\nabla_{{\x^N}}\theta_{N}\left(\y_{\alpha}^{N}(\bv),\bv^{-N}\right)
\end{bmatrix} \right\| + \left\| \alpha \left(\bu-\bv \right)\right\| + \left\|\alpha \left( \y_{\alpha}(\bu) - \y_{\alpha}(\bv)\right)   \right\|.
    \end{align*}
Proceeding term-by-term, first consider 
\begin{align*}
    & \quad \left\|  \sum_{\nu =1}^{N} \left( \left[\nabla_{\x} \theta_{\nu}\left(\bu^{\nu},\bu^{-\nu}\right)- \nabla_{\x} \theta_{\nu}\left(\y_{\alpha}^{\nu}(\bu),\bu^{-\nu}\right)\right] - \left[\nabla_{\x} \theta_{\nu}\left(\bv^{\nu},\bv^{-\nu}\right)- \nabla_{\x} \theta_{\nu}\left(\y_{\alpha}^{\nu}(\bv),\bv^{-\nu}\right)\right] \right) \right\|   \\
    &{\stackrel{\text{triangle ineq.}}{\leq}}   \sum_{\nu =1}^{N} \left\|\left( \left[\nabla_{{\x}} \theta_{\nu}\left(\bu^{\nu},\bu^{-\nu}\right) - \nabla_{{\x}} \theta_{\nu}\left(\bv^{\nu},\bv^{-\nu}\right) \right] - \left[ \nabla_{{\x}} \theta_{\nu}\left(\y_{\alpha}^{\nu}(\bv),\bv^{-\nu}\right) - \nabla_{{\x}} \theta_{\nu}\left(\y_{\alpha}^{\nu}(\bu),\bu^{-\nu}\right)\right]  \right) \right\|   \\
    & \stackrel{\text{{triangle ineq.}}}{\leq}  \sum_{\nu =1}^{N} \left( \left\| \nabla_{{\x}} \theta_{\nu}\left(\bu^{\nu},\bu^{-\nu}\right) - \nabla_{{\x}} \theta_{\nu}\left(\bv^{\nu},\bv^{-\nu}\right)  \right\| + \left\| \nabla_{{\x}} \theta_{\nu}\left(\y_{\alpha}^{\nu}(\bv),\bv^{-\nu}\right) - \nabla_{{\x}} \theta_{\nu}\left(\y_{\alpha}^{\nu}(\bu),\bu^{-\nu}\right) \right\|  \right)    \\
    & \stackrel{\text{$L_{1}^{\nu}$-smoothness of $\theta_{\nu}$}}{\leq}  \sum_{\nu =1}^{N} \left( L_{1}^{\nu} \left\| \bu - \bv  \right\| + \left(L_{1}^{\nu} \sqrt{\|\y_{\alpha}(\bv) - \y_{\alpha}(\bu)\|^2 + \|\bv - \bu\|^2} \right)   \right)  	\\
     & \stackrel{\text{Lips. cont. of   $\y_{\alpha}^{\nu}$}}{\leq}  \sum_{\nu =1}^{N} \left( L_{1}^{\nu} \left\| \bu - \bv  \right\| +  \left(L_{1}^{\nu} \sqrt{1+ {(L_{0}^{\y_{\alpha}})^2}}\right) {\left\| \bu - \bv  \right\|}\right). 
 \end{align*}
Similarly, the second term can be bounded by invoking the Lipschitz continuity of the composition of Lipschitz functions, as shown below for the $\nu$th component.  
\begin{align*}
    \left\|   \nabla_{{\x^{\nu}}}\theta_{{\nu}}\left(\y_{\alpha}^{{\nu}}(\bu),\bu^{-{\nu}}\right) -\fyy{\nabla_{{\x^{\nu}}}}\theta_{{\nu}}\left(\y_{\alpha}^{{\nu}}(\bv),\bv^{-{\nu}}\right)\right\| &\leq L_{1}^{\nu} \sqrt{{(L_{0}^{\y_{\alpha}})}^2\| \bu-\bv\|^2 + \|\bu-\bv\|^2} \\
        & \leq \left(L_{1}^{\nu} \sqrt{1+ {(L_{0}^{\y_{\alpha}})^2}}\right) \|\bu-\bv\|.
\end{align*}
Consequently, {from the preceding inequality we may write} 
\begin{align*}
	\left\| \begin{bmatrix}
        \nabla_{{\x^1}}\theta_{1}\left(\y_{\alpha}^{1}(\bu),\bu^{-1}\right) -\nabla_{{\x^1}}\theta_{1}\left(\y_{\alpha}^{1}(\bv),\bv^{-1}\right) \\
		\vdots \\
        \nabla_{{\x^N}}\theta_{N}\left(\y_{\alpha}^{N}(\bu),\bu^{-N}\right) -\nabla_{{\x^N}}\theta_{N}\left(\y_{\alpha}^{N}(\bv),\bv^{-N}\right)
    \end{bmatrix} \right\| & {\leq \sqrt{\sum_{\nu=1}^N (L_{1}^{\nu})^2 \left(1+ {(L_{0}^{\y_{\alpha}})^{2}}\right) } \|\bu-\bv\|} \\ 
	& \le { \sum_{\nu = 1}^N L_1^{\nu} \sqrt{1 + (L_0^{\y_{\alpha}})^2} \|\bu - \bv \|.} 
\end{align*}
By aggregating the derived bounds, we obtain that for any $\bu, \bv \in \Xscr$, 
    {\begin{align*}
    \left\| \nabla_{{\x}} V_{\alpha}(\bu) - \nabla_{{\x}} V_{\alpha}(\bv)\right\| & \leq 
    \left({2}\sum_{\nu =1}^{N} {\left( L_{1}^{\nu}  + L_{1}^{\nu}  \sqrt{1+ {(L_{0}^{\y_{\alpha}})^2}}\right) }+ \alpha \left( 1 + L_{0}^{\y_{\alpha}}\right) \right)\left\|\bu - \bv   \right\|.
\end{align*}}
\end{proof}
Before proceeding, we document the Lipschitz constants determined thus far in the table below.
\begin{table}[H]
	\centering
	\begin{footnotesize}
	\renewcommand{\arraystretch}{3.0}
	    \rowcolors{2}{gray!25}{white} % Start coloring from the 2nd row (gray!25 for light gray, white for alternate rows)
	\begin{tabular}{|c c l c|}
		\hline
		Map & {Lips.} constant & {Reference} & Value \\
		\hline
		$ \theta^{\nu}(\bullet)$ & $L_{0}^{\nu}$  & {Assump.~\ref{ass:ass-1}}& {Given}  \\
		$\nabla \theta^{\nu}(\bullet)$ & $L_{1}^{\nu}$  & {Assump.~\ref{ass:ass-1}}& {Given}  \\
		%\hline
		$\y_{\alpha}(\bullet)$ & $L_{0}^{\y_{\alpha}}$ & {Lemma~\ref{lem:lips_cont_y_alpha}}& $ \left( \frac{   \alpha+L_G^{\nu} }{  \alpha - L_G^{\nu}} \right)$ for any $\alpha>L_G^{\nu}$ \\
		%\hline
		$V_{\alpha}(\bullet)$ & $L_{0}^{V_\alpha}$ & {Prop.~\ref{prop:V_alpha} (ii)}& ${\displaystyle \sum_{\nu=1}^N \left(L_{0}^{\nu}(1+\sqrt{1+{(L_{0}^{\y_{\alpha}})}^2}) + {4\alpha C^{\nu}} (1+L_{0}^{\y_{\alpha}})\right)} $ \\
		$\nabla_{\x}V_{\alpha}(\bullet)$ & $L_1^{V_{\alpha}}$ & {Prop.~\ref{prop:V_alpha} (iii)} &{\scriptsize $ {\left({2}\sum_{\nu =1}^{N} {\left( L_{1}^{\nu}  + L_{1}^{\nu}  \sqrt{1+ {(L_{0}^{\y_{\alpha}})^2}}\right) }+ \alpha \left( 1 + L_{0}^{\y_{\alpha}}\right) \right)}
  $} \\
		\hline
	\end{tabular}
\end{footnotesize}
\caption{Table Of Lipschitz Constants}
\label{tbl:lipschitz}
\end{table}
\begin{comment}
We summarize the continuity, convexity, and smoothness properties of $V$ and $V_{\alpha}$ in the table below. 
\begin{table}[htp]
	\caption{Continuity, Convexity, Smoothness of $V_{\alpha}$ for $\Xscr$ Satisfying Slater Condition.}\label{table:vcx-smooth-V-table}
	For the purpose of the following table, suppose that $\theta_{\nu}$ is continuously differentiable for $\nu = 1, \dots \N$. That is, $\theta_{\nu} \in C^1.$ Since we are assuming that $\Xscr$ satisfies a Slater condition, $V$ and $V_{\alpha}$ are both $C^0$. 
	\begin{center}
		\begin{tabular}{|c|c|c|c|c|}
			\hline
			  \multirow{2}{*}{} &\multicolumn{2}{c|}{$V$} & \multicolumn{2}{c|}{$V_{\alpha}$}  \\
			  \cline{2-5}
			&  $C^1$ & convex &  $C^1$ & convex \\ 
			\hline
			\multirow{1}{*}{$\theta_{\nu}$ is convex}  & \xmark  & \cmark & \cmark & \cmark   \\
			\cline{2-5}
			 \cline{2-5}
			 \multirow{1}{*}{$\theta_{\nu}$ is player-convex}   & \xmark &\xmark& \cmark & \xmark \\
			\cline{2-5}
			\hline
		\end{tabular}
	\end{center}
\end{table}
\end{comment}

  It is crucial to note that
$\y_{\alpha}(\x)$ is a vector, where the
$\nu^{\rm th}$ block-coordinate of
$\y_{\alpha}{(\x)}$, denoted by
$\y_{\alpha}^{\nu}{(\x)}$, is an
\textit{exact solution to problem
\eqref{prob-subprob}}. Armed with such a
solution, one has a closed-form expression
for the gradient and may apply a gradient
descent scheme to find the optimal solution.
However, it is not difficult to imagine
settings where exact solutions to the
problem are not available. For example,
absent a closed-form solution for
$\y_{\alpha}(\x)$, any numerical scheme will
result in some imprecision; {specifically,} constraints on
computational resources or time will lead to
inexactness {while} stochastic programs {defined on general probability spaces} may be
impossible to resolve exactly in finite
time.  We therefore consider the setting
where the problem only admits inexact
solutions in finite time and study how the
inexactness in a solution of the subproblem
of computing $\y_{\alpha}(\x)$ affects the
convergence of a minimization scheme applied
to $V_{\alpha}(\bullet).$

\section{An Inexact Algorithm For Computing Equilibria} \label{sec:NI_inexact}
%We begin by recalling some prior approaches that have employed the {\bf NI} function  as part of their computational framework.  In \cite{krawczyk}, relaxation algorithms for computing equilibria are introduced which employ the best-response.
%In \cite{kanzow}, three optimization problems relevant to the computation of equilibria are introduced and explored. A condition is presented that may be used to confirm that a stationary point of $V_{\alpha}$ is indeed a global Nash equilibrium. More recently, in \cite{raghunathan}, the authors introduce a Gradient-based function described above and demonstrate that a gradient descent algorithm applied to this function converges sublinearly to a stationary point. 
%Guarantees of convergence are provided for a standard gradient descent approach applied. An application to the river basin pollution problem is discussed. \\
%In this section, we present papers for further reading on computation of equilibria and NI-functions. \\
%\cite{faccandpang} Provides significant background on the variational approach to computing Nash equilibria. The NI function is defined and discussed. \\
%In this section, we formalize the assumptions underlying our problem of interest. We then introduce the notion of an inexact solution to the problem of resolving $\y_{\alpha}(\x)$. Equipped with these inexact solutions, we present an algorithm for computing stationary points of the regularized {\bf NI} function.
%\subsection{Problem properties and assumptions}
Our focus for the remainder of this paper is on resolving the optimization problem {defined as} 
\begin{align}\label{eqn:min_V_alpha}
	\min_{\x \in \Xscr}\, V_{\alpha}(\x),
\end{align}
where  $\theta_{\nu}\left(\x^{\nu},\x^{-\nu}\right)$ is {expectation-valued and} given by \eqref{eqn:nash_opt} for {every} $\nu \in
\{1, \dots, N\}$. In the prior section, we have derived Lipschitz constants for $V_{\alpha}$ and its gradient. This analysis allows us to develop an inexact gradient framework; note that an exact gradient requires computing an exact solution $\y^{\nu}(\x)$ for any $\nu \in \{1, \cdots, N\}$ and any $\x \in \Xscr$. Furthermore, the convergence guarantees for our scheme ensure that the gradient-based framework generates a sequence that converges to a stationary point.  In \cite{kanzow}, a condition is given that ensures that a {feasible} stationary point of $V_{\alpha}$ is {indeed} a Nash equilibrium. We state {this condition next}. 
\begin{proposition}[{\bf A Sufficient Condition For A Stationary Point to be an NEP}~\cite{kanzow}] 
	\label{ass-optimality}\em
	\noindent Consider problem $V_{\alpha}$ as defined in \eqref{def:valpha}. For a given ${\x} \in \Xscr$, with ${\x} \neq \y_{\alpha}\left( \x \right)$, the inequality
	\begin{align}\label{eqn:kanzow_assumption_one}
		\left(\sum_{\nu =1}^{N}\left[\nabla \theta_{\nu}\left(\x^{\nu},\x^{-\nu}\right)- \nabla \theta_{\nu}\left(\y_{\alpha}^{\nu}{\left( \x \right)},{\x}^{-\nu}\right)\right]\right)^{\top} \left(\x - \y_{\alpha}\right( \x \left) \right) > 0 \, 
	\end{align}
	holds. Suppose  $\x^*$ is a stationary point of \fyy{\eqref{eqn:min_V_alpha}}. Then ${\x^*}$ is a Nash equilibrium if  \eqref{eqn:kanzow_assumption_one} holds at $\x^*.$ $\hfill$ $\Box$  
\end{proposition}
% In fact, under a suitable condition, this stationary point is indeed a Nash equilibrium.
 In Section~\ref{sec:inexact}, we discuss the properties of the inexact solution of $\y^{\nu}(\x)$ for any $\nu$. We conclude with Section~\ref{sec:algo}, where we present the overall gradient-based framework as well as the stochastic approximation scheme for computing an inexact solution.
 \subsection{\texorpdfstring{Inexact computation of $\y^{\nu}_\alpha(\x)$}{Inexact computation of yᵛᵅ(x)}} \label{sec:inexact}
 We begin by introducing the notion of an inexact solution to the strongly convex stochastic optimization problem that arises from computing $\y_{\alpha}(\x)$. 
\begin{definition}[{\bf Inexact solution}] \label{def:inexact_solution}
	\em
	Let $\y_{\alpha}(\x) = (\y_{\alpha}^{1}(\x), \dots, \y_{\alpha}^N(\x))$ be an exact solution to problem~\eqref{prob-subprob}. %Let $\y_{\epsilon}:\mathbb{R}^n \times \Omega \to \mathbb{R}^n$ denote a random variable associated with the probability space $(\Omega, \mathcal{F},\mathbb{P})$. 
	The random  variable  $\y_{\alpha,\epsilon}(\bullet)$ is an $\epsilon$-approximate solution to problem~\eqref{prob-subprob} if for any $\x \lm{\in \Xscr}$, the following holds almost surely. 
		$$\mathbb{E}\left[ \| \y_{\alpha,\epsilon}(\x)- \y_{\alpha}(\x) \|^2 \, \mid \, \x \right] \leq \epsilon. \hspace{1in} \Box $$
	%We will call $\y_{\epsilon}(\x,\omega)$ unbiased if $\mathbb{E}\left[\y_{\epsilon}(\x,\omega) - \y_{\alpha}(\x) \right] =0$. 
\end{definition}

In this framing, we regard $\y_{\alpha,\epsilon}(\x)$ {as} a random variable.  {Based on this inexact solution $\y_{\alpha,\epsilon}(\x)$}, {we employ a sample-average approximation of $\nabla V_{\alpha}(\x)$ with $M$  samples of the gradient estimator for each player in computing an estimator $\nabla V_{\alpha,\epsilon,M}(\x)$, } defined as
\begin{align}\notag
    \nabla V_{\alpha, \epsilon,M}\left( \x\right) &\triangleq \sum_{\nu =1}^{N}\left[\tfrac{\sum_{j=1}^M \left(\nabla \tilde{\theta}_{\nu}\left(\x^{\nu},\x^{-\nu},{\xi_j^{\nu}}\right)- \nabla \tilde{\theta}_{\nu}\left(\y_{{\alpha},\epsilon}^{\nu}{(\x)},\x^{-\nu},{\xi_j^{\nu}}\right)\right)}{M}\right]   \\
\label{def:nablaV}
		& + \pmat{
            \tfrac{\sum_{j=1}^M\nabla_{x^1}\tilde{\theta}_{1}\left(\y_{{\alpha,}\epsilon}^{1}({\x}),\x^{-1},{\xi_j^1}\right)}{M} -\alpha(\x^1- {\y_{\alpha,\epsilon}^1(\x)})\\
				\vdots \\
                \tfrac{\sum_{j=1}^M\nabla_{\x^N}\tilde{\theta}_{N}\left(\y_{{\alpha},\epsilon}^{N}({x}),\x^{-N},{\xi_j^N}\right)}{M}-\alpha(\x^N- {\y_{\alpha,\epsilon}^N(\x)})}, 
\end{align}
may be used as an estimator of $\nabla V_{\alpha}(\x)$ in an optimization scheme. 
\begin{definition} \label{def:stoch_errors2}
 The difference between our inexact mini-batch {gradient estimator} and the true gradient of $V_{\alpha}{(\bullet)}$ at $\x$ {is denoted by ${\bf e}_{\epsilon,M}(\x)$, defined as} 
\begin{align}
	{\bf e}_{\epsilon,M}{(\x)} \, \triangleq \, {\nabla V_{\alpha, \epsilon,M}(\x)} -\nabla V_{\alpha}(\x). 
\end{align}
\end{definition}
Before proceeding, we require the following assumption. 
\begin{assumption}[{\bf Bias and moment assumptions}]\em
    {There exist{s a} positive scalar $\sigma$ such that
    for all $\x \in \Xscr,$ the following hold almost surely}. \\
\noindent {\bf (i)} $
	{ \mathbb{E}\left[\nabla {\theta}_{\nu}\left(\x^{\nu},\x^{-\nu}\right)- \nabla \tilde{\theta}_{\nu}\left(\x^{\nu},\x^{-\nu},\bxi \right)   \mid \x\right]} = 0 $;
	
\noindent {\bf (ii)} ${ \mathbb{E}\left[\left\|\nabla {\theta}_{\nu}\left(\x^{\nu},\x^{-\nu}\right)- \nabla \tilde{\theta}_{\nu}\left(\x^{\nu},\x^{-\nu},\bxi\right)  \right\|^2 \mid \x\right]} \leq {{\sigma}^2}$. $\hfill \Box$  

%\noindent {\bf (iii)} $
%	{ \mathbb{E}\left[\left\|\nabla \tilde{\theta}_{\nu}\left(\x^{\nu},\x^{-\nu},\bxi\right)  \right\|^2 \mid \x\right]} \leq B^2  $.
\end{assumption}
Then we may derive moment properties on ${\bf e}_{\epsilon,M}(\x)$. 
\begin{lemma}[{\bf Moment properties of ${\bf e}_{\epsilon,M}{(\x)}$}]\label{lem:stoch_error_var}\em
    Consider Definition \ref{def:stoch_errors2}. Suppose Assumption~\ref{ass:ass-1} holds. Then {for any $\x$},  
	% \noindent (i)  $\mathbb{E}[e_{\alpha,\epsilon}  \mid \x] = 0$ if $\y_{\epsilon}(\x,\omega)$ unbiased.\\
    $\mathbb{E}[\|{\bf e}_{\epsilon,M}(\x) \|^2 \mid \x] \leq  \tfrac{{(3N + 2){N}\sigma}^2}{M}  + \left( \left({2N+2} \right)\sum_{\nu =1}^{N} (L_{1}^{\nu})^2 + 4 \alpha \right) \epsilon$ {holds almost surely}.
%	for some $\nu < \infty.$
\end{lemma} 
\begin{proof}
{For ease of exposition,  let $\y_{\alpha}(\x)$ be denoted by $\y_{\alpha}$ and $\y_{\alpha,\epsilon}(\x)$ be  given by $\y_{\epsilon}$.}
	{To derive a bound on $\mathbb{E}[\|{\bf e}_{\epsilon,M}{(\x)}\|^2\mid \mathcal \x]$, we may express this conditional expectation as follows.} 
	\begin{align*}
		&\quad \mathbb{E}[\|{\bf e}_{\epsilon,M}{(\x)}\|^2 \mid \x] 
		= \mathbb{E}\left[ \left\|\nabla V_{\alpha, \epsilon,M}(\x) -\nabla V_{\alpha}(\x) \right\|^2 {\mid \x}\right]\\
		&=  \mathbb{E}\left[\left\|\sum_{\nu =1}^{N}\left[{\tfrac{\sum_{j=1}^M \nabla \tilde{\theta}_{\nu}\left(\x^{\nu},\x^{-\nu},{{\bxi}^{\nu}_j}\right)}{M}}- \tfrac{\sum_{j=1}^M \nabla \tilde{\theta}_{\nu}\left(\y_{\epsilon}^{\nu},\x^{-\nu},{{\bxi}^{\nu}_j}\right)}{M}\right] + 
		\begin{bmatrix}
				\tfrac{\sum_{j=1}^M\nabla_{\x^1}\tilde{\theta}_{1}\left(\y_{{\epsilon}}^{1},\x^{-1},{\bxi^1_j}\right)}{M} \\
				\vdots \\
				\tfrac{\sum_{j=1}^M \nabla_{\x^N}\tilde{\theta}_{N}\left(\y_{{\epsilon}}^{N},\x^{-N},{\bxi^N_j}\right)}{M}
		\end{bmatrix} - \alpha \left(\x- \y_{{\epsilon}} \right)\right. \right.   \\
		& - \left. \left. \left( {\sum_{j=1}^M \left[\nabla  {\theta}_{\nu}\left(\x^{\nu},\x^{-\nu}\right)- \nabla  {\theta}_{\nu}\left(\y_{\alpha}^{\nu},\x^{-\nu}\right)\right] } +  
		\begin{bmatrix}
			\nabla_{\x^1}\theta_{1}\left(\y_{\alpha}^{1},\x^{-1}\right) \\
			\vdots \\
			\nabla_{x^N}\theta_{N}\left(\y_{\alpha}^{N},\x^{-N}\right)
		\end{bmatrix}
		 - \alpha \left(\x- \y_{\alpha} \right) \right) \right\|^2 \bigg| \x \right]\\
		&=  \mathbb{E}\left[\left\| \sum_{\nu =1}^{N}\left[\nabla {\theta}_{\nu}\left(\x^{\nu},\x^{-\nu}\right)- \tfrac{\sum_{j=1}^M \nabla \tilde{\theta}_{\nu}\left(\x^{\nu},\x^{-\nu},{\bxi^{\nu}_j}\right)}{M}\right] + \sum_{\nu =1}^{N}\left[\nabla {\theta}_{\nu}\left(\y_{\alpha}^{\nu},\x^{-\nu}\right)- \tfrac{\sum_{j=1}^M \nabla \tilde{\theta}_{\nu}\left(\y_{\epsilon}^{\nu},\x^{-\nu},{\bxi^{\nu}_j}\right)}{M}\right]  \right. \right. 
\end{align*}
\begin{align*}
		& +\left. \left.\left( \begin{bmatrix}
			\nabla_{x^1}\theta_{1}\left(\y_{{\alpha}}^{1},\x^{-1}\right) \\
			\vdots \\
			\nabla_{x^N}\theta_{N}\left(\y_{{\alpha}}^{N},\x^{-N}\right)
		\end{bmatrix} - {\begin{bmatrix}
				\tfrac{\sum_{j=1}^M\nabla_{x^1}\tilde{\theta}_{1}\left(\y_{\epsilon}^{1},\x^{-1},{\bxi^1_j}\right)}{M} \\
				\vdots \\
				\tfrac{\sum_{j=1}^M \nabla_{x^N}\tilde{\theta}_{N}\left(\y_{\epsilon}^{N},\x^{-N},{\bxi^N_j}\right)}{M}
		\end{bmatrix}}+ \alpha \left(\y_{\epsilon} - \y_{\alpha} \right) \right) \right\|^2 \, \bigg| \, \x \right].
	\end{align*}
	Applying the triangle inequality, 
	\begin{align*}
		\mathbb{E}[\|{\bf e}_{\epsilon,M}{(\x)}\|^2 \mid \x] & \leq
		4\mathbb{E}\left[\left\|\sum_{\nu =1}^{N}\left[\nabla {\theta}_{\nu}\left(\x^{\nu},\x^{-\nu}\right)- \tfrac{\sum_{j=1}^M \nabla \tilde{\theta}_{\nu}\left(\x^{\nu},\x^{-\nu},{\bxi_j^{\nu}}\right)}{M}\right]  \right\|^2 \mid \x\right]\\
		&+ 4\mathbb{E}\left[ \left\|\sum_{\nu =1}^{N}\left[\nabla {\theta}_{\nu}\left(\y_{\alpha}^{\nu},\x^{-\nu}\right)- \tfrac{\sum_{j=1}^M \nabla \tilde{\theta}_{\nu}\left(\y_{\epsilon}^{\nu},\x^{-\nu},{\bxi_j^{\nu}}\right)}{M}\right] \right\|^2 \mid \x\right]  \\
		& + 4\mathbb{E}\left[\left\|\begin{bmatrix}
			\nabla_{\x^1}\theta_{1}\left(\y_{{\alpha}}^{1},\x^{-1}\right) \\
			\vdots \\
			\nabla_{\x^N}\theta_{N}\left(\y_{{\alpha}}^{N},\x^{-N}\right)
		\end{bmatrix} - {\begin{bmatrix}
				\tfrac{\sum_{j=1}^M\nabla_{\x^1}\tilde{\theta}_{1}\left(\y_{{\epsilon}}^{1},\x^{-1},{\bxi_j^1}\right)}{M} \\
				\vdots \\
				\tfrac{\sum_{j=1}^M \nabla_{\x^N}\tilde{\theta}_{N}\left(\y_{{\epsilon}}^{N},\x^{-N},{\bxi_j^N}\right)}{M}
		\end{bmatrix}}\right\|^2 \bigg| \x \right] \\
		&+ 4\alpha^2 \mathbb{E}\left[\left\|\y_{\epsilon} - \y_{\alpha}  \right\|^2 \bigg| \x \right].
	\end{align*}
	Proceeding term-by-term, we first observe that 	\begin{align*}
        & \quad \mathbb{E}\left[\left\|\sum_{\nu =1}^{N}\left[\nabla {\theta}_{\nu}\left(\x^{\nu},\x^{-\nu}\right)- \tfrac{\sum_{j=1}^M \nabla \tilde{\theta}_{\nu}\left(\x^{\nu},\x^{-\nu},{\bxi_j^{\nu}}\right)}{M}\right]  \right\|^2 \mid \x\right] \\ 
        &\leq {\sum_{\nu =1}^{N} N \mathbb{E}\left[\left\|\left[\nabla {\theta}_{\nu}\left(\x^{\nu},\x^{-\nu}\right)- \tfrac{\sum_{j=1}^M \nabla \tilde{\theta}_{\nu}\left(\x^{\nu},\x^{-\nu},{\bxi_j^{\nu}}\right)}{M}\right]  \right\|^2 \mid \x\right]}
\leq \tfrac{{N^2}{\sigma}^2}{M}, 
	\end{align*} holds almost surely. The second term can be bounded {in a similar fashion, as shown next.} 
	\begin{align*}
		& \quad \mathbb{E}\left[ \left\|\sum_{\nu =1}^{N}\left[\nabla {\theta}_{\nu}\left(\y_{\alpha}^{\nu},\x^{-\nu}\right)- \tfrac{\sum_{j=1}^M \nabla \tilde{\theta}_{\nu}\left(\y_{\epsilon}^{\nu},\x^{-\nu},{\bxi^{\nu}_j}\right)}{M}\right] \right\|^2 \mid \x\right]  \\
		& \leq \sum_{\nu =1}^{N}{N} \mathbb{E}\left[ \left\|\left[\nabla {\theta}_{\nu}\left(\y_{\alpha}^{\nu},\x^{-\nu}\right)- \tfrac{\sum_{j=1}^M \nabla \tilde{\theta}_{\nu}\left(\y_{\epsilon}^{\nu},\x^{-\nu},{\bxi_j^{\nu}}\right)}{M}\right] \right\|^2 \mid \x\right]  \\
		& \leq \sum_{\nu =1}^{N}{N} {\mathbb{E}\left[ 2\left\|\nabla {\theta}_{\nu}\left(\y_{\alpha}^{\nu},\x^{-\nu}\right)-\nabla {\theta}_{\nu}\left(\y_{\epsilon}^{\nu},\x^{-\nu}\right) \right\|^2 +2 \left\|\nabla {\theta}_{\nu}\left(\y_{\epsilon}^{\nu},\x^{-\nu}\right)  -\tfrac{\sum_{j=1}^M \nabla \tilde{\theta}_{\nu}\left(\y_{\epsilon}^{\nu},\x^{-\nu},{\bxi_j^{\nu}}\right)}{M} \right\|^2 \mid \x\right]}  \\
		& \leq {2N} \sum_{\nu =1}^{N} {(L_{1}^{\nu})^2} \mathbb{E}\left[\| \y^{\nu}_{\alpha} - \y^{\nu}_{\epsilon}\|^2 \, \mid \, \x\right] + {2N} \sum_{\nu =1}^{N}\mathbb{E}\left[ \left\| \nabla {\theta}_{\nu}\left(\y_{\epsilon}^{\nu},\x^{-\nu}\right)- \tfrac{\sum_{j=1}^M \nabla \tilde{\theta}_{\nu}\left(\y_{\epsilon}^{\nu},\x^{-\nu},{\bxi_j^{\nu}}\right)}{M} \right\|^2 \mid \x\right]  \\
		& \leq {2N} \sum_{\nu =1}^{N} {(L_{1}^{\nu})^2} \mathbb{E}\left[\| \y^{\nu}_{\alpha} - \y^{\nu}_{\epsilon}\|^2 \, \mid \, \x\right] + \tfrac{{2N^2\sigma^2}}{M}{,}
	\end{align*}
	where the {second} inequality invokes the Lipschitzian bound
	$\| \nabla \theta_{\nu}\left(\y_{\alpha}^{\nu},\x^{-\nu}\right)- \nabla \theta_{\nu}\left(\y_{\epsilon}^{\nu},\x^{-\nu}\right)\|^2 \leq {{\left(L_{1}^{\nu}\right)}^2} \|\y_{\alpha}^{\nu} - \y_{\epsilon}^{\nu} \|^2.$ The third term can be bounded as 
	\begin{align*}
& \quad 		\mathbb{E}\left[\left\|\begin{bmatrix}
			\nabla_{x^1}\theta_{1}\left(\y_{{\alpha}}^{1},\x^{-1}\right) \\
			\vdots \\
			\nabla_{x^N}\theta_{N}\left(\y_{{\alpha}}^{N},\x^{-N}\right)
		\end{bmatrix} -  {\begin{bmatrix}
				\tfrac{\sum_{j=1}^M\nabla_{x^1}\tilde{\theta}_{1}\left(\y_{{\epsilon}}^{1},\x^{-1},{\bxi_j^1}\right)}{M} \\
				\vdots \\
				\tfrac{\sum_{j=1}^M \nabla_{x^N}\tilde{\theta}_{N}\left(\y_{{\epsilon}}^{N},\x^{-N},{\bxi_j^N}\right)}{M}
		\end{bmatrix}}\right\|^2 \bigg| \x \right] \\
& \leq \sum_{\nu=1}^N \mathbb{E}\left[ \left\|\nabla_{x^\nu}\theta_{\lm{\nu}}\left(\y_{{\alpha}}^{\nu},\x^{-\nu}\right) -
				\tfrac{\sum_{j=1}^M\nabla_{x^\nu}\tilde{\theta}_{\nu}\left(\y_{{\epsilon}}^{\nu},\x^{-\nu},{\bxi_j^{\nu}}\right)}{M} \right\|^2 \mid \x \right]\\
	& \leq {\sum_{\nu =1}^{N} 2{(L_{1}^{\nu})^2} \mathbb{E}\left[\| \y^{\nu}_{\alpha} - \y^{\nu}_{\epsilon}\|^2 \, \mid \, \x\right] + \tfrac{2N\sigma^2}{M}.}
	\end{align*} 
	Putting this all together, we have 
	{\begin{align*}
		\mathbb{E}[\|{\bf e}_{\epsilon,M}(\x)\|^2 \mid \x] &  \leq \tfrac{{N}{(3N + 2)\sigma}^2}{M}  + 2 \left({N+1} \right)\sum_{\nu =1}^{N} {(L_{1}^{\nu})^2} \mathbb{E}\left[\| \y^{\nu}_{\alpha} - \y^{\nu}_{\epsilon}\|^2 \, \mid \, \x\right]  +   4 \alpha ^2 \mathbb{E}\left[\left\| \y_{\alpha} - \y_{\epsilon}\right\|^2 \, \mid \, \x \right] \\
		& \leq \tfrac{{(3N + 2){N}\sigma}^2}{M}  + \left( \left({2N+2} \right)\sum_{\nu =1}^{N} {(L_{1}^{\nu})^2} + 4 \alpha \right) \epsilon   \quad \mbox{almost surely.}
	\end{align*}}
\end{proof}
\subsection{Algorithm description} \label{sec:algo}
{In this section, we present} a {Monte-Carlo sampling enabled} inexact gradient scheme, {formally defined as Algorithm}~\ref{algorithm:NI_inexact}. {As the reader may observe, this is a relatively simple projected gradient scheme where an inexact gradient estimator $\nabla V_{\alpha}(\x_k)$ is computed at iteration $k$. After $K$ steps, the algorithm returns a random variable $\x_{R_{\ell,K}}$ for which rate and complexity guarantees may be provided. The scheme takes the steplength sequence $\{\gamma_k\}$, the sample-size sequence $\{M_k\}$, and the inexactness sequence $\{\epsilon_k\}$ as inputs.} 
	\begin{algorithm}[htb]
        \caption{{{\bf NI}-based Sampling-enabled Inexact Gradient Method}}\label{algorithm:NI_inexact}
		%\begin{boxedminipage}{165mm}
		{   \begin{algorithmic}[1]
				\STATE\textbf{input:}  Given $\x_0 \in \Xscr$,  prescribed inexactness sequence $\{\epsilon_k\} > 0$, batch-size sequence $\{M_k\}$, and stepsize sequence $\{\gamma_k\}>0$,  such that $\sum_{k=1}^{\infty}\gamma_{k}^2 < \infty$, $K \in \mathbb{N},$
				\FOR {$k=0,1,\ldots,{K}-1$}
                \STATE (i) Generate $\epsilon_k-$approximate solution $\y_{{\epsilon_k}}(\x_{k})$ via {Algorithm} \ref{algorithm:stochastic_approximation}.  
				% \STATE Generate a random {batch $\omega_{j,k}$ for $j=1,\ldots,N_k$}  
				\STATE (ii) Calculate $\nabla V_{\alpha, {\epsilon_k},M_k}(\x_k)$ {via \eqref{def:nablaV}}.   
				\STATE (iii) {Update $\x_k$ via the update rule:}
				$\x^{k+1}:= \Pi_{\Xscr}\left[ \x_{k}-\gamma_k \nabla V_{\alpha, \epsilon_k,M_k}(\x_k) \right]$
				\ENDFOR
				\STATE Return $\x_{R_{\ell, K}} \, .$ {(cf. Lemma~\ref{lemma:bound_res})}
		\end{algorithmic}}
	\end{algorithm}
        {In step (i) of the aforementioned scheme, $\y_{\epsilon_k}(\x_k)$ is computed. We observe that this is an $\epsilon$-solution to a strongly convex optimization problem. Such a solution is obtained by employing a stochastic approximation scheme as captured in Algorithm~\ref{algorithm:stochastic_approximation}. We provide a precise rate statement that prescribes the minimum number of steps $T_k$ required to obtain an $\epsilon_k$-solution.}   
		\begin{algorithm}[htb]
		\caption{Stochastic Approximation {Method} For Computing Inexact {Solution} $\y_{\alpha,{\epsilon_k}}(\x_{k})$}\label{algorithm:stochastic_approximation}
		%\begin{boxedminipage}{165mm}
		   \begin{algorithmic}[1]
            \STATE\textbf{input:}  Given $\bz_0 \in \Xscr$, {simulation length ${T_k^{\nu}}$}, and stepsize sequence $\{\beta_i\}>0$;
				\FOR {$\nu = 1,\ldots, N$}
				\FOR {$i=0,1,\ldots,{{T_k^{\nu}}}-1$} 
			%	\STATE (i) Using mini batch \lm{$\{\nabla_{x^{\nu}}\tilde{\theta}_{\nu}(\bz_{i}^{\nu}, \x_k^{-\nu},\xi_{j,i} )\}_{j=1}^{T_k}$}, compute the stochastic approximation $ \nabla_{\bz_{\nu}}\hat{\theta}_{\nu}(\bz_i) := \tfrac{\sum_{j=1}^{T_k} \nabla_{\bz_{\nu}} \tilde{\theta}_{\nu}\left(\bz_{i}^{\nu},\x_{k}^{-\nu},\xi_{j,i} \right)}{T_k}$.
				\STATE  {Update ${\bz_i^{\nu}}$ via the update rule:}
                $\bz_{i+1}^{\nu}:= \Pi_{\Xscr^{{\nu}}}\left[ \, \bz_{i}^{\nu}-\beta_i {\big(}{\nabla_{x^{\nu}}\tilde{\theta}_{\nu}(\bz_{i}^{\nu}, \x_k^{-\nu},\xi_{k,i} )} {+ \alpha \left(\bz_{i}^{\nu} - \x_k^{\nu}\right)}{\big)} \, \right]$
				\ENDFOR
				\ENDFOR
				\STATE Return ${\y_{\alpha,\epsilon_k}(\x_k) = \pmat{\bz^{\nu}_{T^{\nu}_k}}_{\nu=1}^N}.$
		\end{algorithmic}
	\end{algorithm}
            Such schemes have been examined in some detail in a series of monographs (cf.~\cite{Kush03,Borkar08,shapiro09lectures}). {Our rate statement relies on the following definition.}
\begin{definition}[\bf Diameter of compact set]\em
	Let $\Xscr$ be a compact subset $\R^n \, .$ Then the diameter of $\Xscr,$ denoted by {${\cal D}_\Xscr$}, is defined as 
	\begin{align}
		\Dscr_{\Xscr} \triangleq \sup_{\uu,\vv \in \Xscr} \|\uu-\vv\|^2 \, .
	\end{align}
\end{definition}
\begin{proposition}[{\cite[Ch.~5, Eq.~296]{shapiro09lectures}}]\em
    Suppose {Assumption} \ref{ass:ass-1}
holds. {For $\nu = 1, \cdots, N$}, let $\{\bz^{{\nu}}_i\}$ be
generated by
Algorithm~\ref{algorithm:stochastic_approximation}
{where} $\bz_0 = \x$, {$$T_k^\nu { = } 
{\Bigl\lceil \epsilon_k^{-1} \Big (
\tfrac{2\left({L_{1}^{V_{\alpha}}}\right)^2}{\alpha^2}
+ 2\Dscr_{\Xscr^{\nu}} \Big) \Bigr\rceil},$$}
{and} $\{\beta_i\}$ chosen so that ${\beta_i} <
\tfrac{1}{2 \alpha i}$. Then {$\y_{\alpha,\epsilon_k}(\x) = \pmat{\bz_{T_k^\nu}^{\nu}}_{\nu=1}^N$}
is an ${\epsilon_k}-${approximation of} $\y_{\alpha}(\x)$, {\fyy{i.e.,} } 
	\begin{align}
        \mathbb{E}\left[ \, \| {\y_{\alpha,\epsilon_k}(\x_k)} - \y_{\alpha}({\x_k}) \|^2 \, \mid \, {\x_k} \right] \,\le \,  {\epsilon_k} \quad \mbox{almost surely}. 
	\end{align}
%\lm{{\bf Uday-- this requires the assumption that $\mathbb{E}[\theta_{\nu}(\x)] < M$ for all $\x \in \Xscr $ }}
\end{proposition}
%\uvs{In particular, observe that $\uvs{T_k^{\nu}} = \mathcal{O}\left( \uvs{\epsilon_k}^{-1}\frac{(L_1^{V_\alpha})^2}{\alpha^2} \right).$}
%\begin{proof} Omitted, see \cite{lei2020asynchronous}. 	\end{proof}

\section{Convergence and rate analysis}\label{sec:analysis}
%\begin{assumption}[{\bf Assumptions for nonconvex $V_{\alpha}$}] 
%	\label{ass:nonconvex-V}\em
%	\noindent Consider problem $V_{\alpha}$ as defined in \eqref{def:valpha}.
	
%	\noindent (i) $V_{\alpha}$ is $L_0$-Lipschitz on an open set containing $\Xscr$ and $L_1$-smooth on $\Xscr$.
	
%	\noindent (ii) $\Xscr \subseteq \R^n$ is a nonempty, compact, and convex set.
	
%	\noindent (iii) $\theta_{\nu}$ satisfies \ref{ass:player-convexity-assumption} and $L_1-$smooth on and open set containing $\Xscr$ for $\nu \in \{1, \dots N\}.$ 
	
%	\noindent (iv) Let $\epsilon'= \frac{20N\nu^2}{M} + (16 L_1^2+4\alpha^2)\epsilon .$ There exists a random variable $\y_{\epsilon}(\x, \xi)$ such that $\y_{\epsilon}(\x, \xi)$ is an $\epsilon'-$ approximate solution of $\y_{\alpha}(\x)$ for all $\x \in \Xscr$. 
%end{assumption}
    We define a common residual used to measure the departure from stationarity
    of an $L$-smooth function {with respect to a closed and
    convex set}~\cite{beck14introduction}. We also recall its inexact
    counterpart and show that the the norm-squared of the error-afflicted
    residual can be bounded in terms of the true residual and the magnitude of
    the error~\cite{CSY2021MPEC}.  
\begin{definition}[{\bf The residual
mapping}]\label{def:res_maps}\em
Given ${\beta,\epsilon, M >0}$, for any $\x \in \mathbb{R}^n$ and $ {\bf e}_{\epsilon, M} \in \mathbb{R}^n$ an arbitrary given vector,
    let the residual mappings $G_{\beta}(\x)$ and $\tilde{G}_{\beta,\epsilon,{M}}(\x)$ be defined as 
    \begin{align}
        G_{\beta}(\x) &\triangleq \beta {\left(\x - \Pi_{\Xscr}\left[\x - \tfrac{1}{\beta} \nabla_{{\x}} {V_{\alpha}(\x)}\right]\right)} \mbox{ and }\\
		\hspace{-0.05in}\tilde{G}_{\beta,\epsilon,M}(\x) &\triangleq \beta\left( \x - \Pi_{\Xscr}\left[\x - \tfrac{1}{\beta} ({\nabla_{{\x}} {V_{\alpha,\epsilon,M}}(\x)}) \right]\right) 
         = \beta\left( \x - \Pi_{\Xscr}\left[\x - \tfrac{1}{\beta} ({\nabla_{{\x}} {V_{\alpha}}(\x)}+{\bf e}_{\epsilon,M}(\x)) \right]\right). 
\end{align} $\hfill$ $\Box$
\end{definition}
 Note that ${\nabla_x V_{\alpha}(\x)}+{\bf e}_{\epsilon,M}(\x)$ in the definition of $\tilde{G}_{\beta,\epsilon,M}(\x)$ is meant to correspond to an error-afflicted or  an inexact estimate of the gradient of $V_{\alpha}(\x).$
    We now {{recall} a result, first shown in ~\cite{CSY2021MPEC}}, that will be useful in proving the convergence of {Algorithm}~\ref{algorithm:NI_inexact} in the case that $V_{\alpha}$ is nonconvex. 
    \begin{lemma}\label{bd_G}
	At any point $\x \in \Xscr$, the following inequality holds. 
    \begin{align*}
       \|G_{\beta}{(\x)}\|^2 \leq {2} \|\tilde{G}_{\beta, \epsilon,M}{(\x)}\|^2 + 2 \|{ \bf e}_{ \epsilon,M} (\x)\|^2.
    \end{align*}
    \begin{proof}
        Invoking the definition of $G_{\beta}(\x)$, observing that $V_{\alpha,\epsilon,M}(\x) = V_{\alpha}(\x) +{\bf e}_{\epsilon,M}  $, and adding and subtracting  $\beta\Pi_{\Xscr}\left[\x - \tfrac{1}{\beta} ({\nabla_{{\x}} V_{\alpha,\epsilon,{M}}(\x)}) \right]$, we have
        \begin{align*}
            G_{\beta}(\x) & = \left\|  \beta \left(\x - \Pi_{\Xscr}\left[\x - \tfrac{1}{\beta} \nabla_{\x} {V_{\alpha}(\x)}\right]\right) \right\|^2 \\
			              %& = \lm{\left\|  \beta \left(\x - \Pi_{\Xscr}\left[\x - \tfrac{1}{\beta} \nabla_{\x} {V_{\alpha,\epsilon,{M}}(\x)} + {\bf e}_{\epsilon,{M}}\right] \right) \right\|^2} \\
             & = \left\| \beta\left( \x - \Pi_{\Xscr}\left[\x - \tfrac{1}{\beta} ({\nabla_{{\x}} V_{\alpha,\epsilon,{M}}(\x)}) \right]\right) + \beta \Pi_{\Xscr}\left[\x - \tfrac{1}{\beta} \nabla_{{\x}} {V_{\alpha}(\x)}\right] -  \beta\Pi_{\Xscr}\left[\x - \tfrac{1}{\beta} ({\nabla_{{\x}} V_{\alpha,\epsilon,{M}}(\x)}) \right] \right\|^2 \\
             & \leq  {2} \left\| \beta\left( \x - \Pi_{\Xscr}\left[\x - \tfrac{1}{\beta} ({\nabla_{{\x}} V_{\alpha,\epsilon,M}(\x)}) \right]\right)\right\|^2 \\
	& + 2\left\| \beta \Pi_{\Xscr}\left[\x - \tfrac{1}{\beta} \nabla_{{\x}} {V_{\alpha}(\x)}\right] -\beta\Pi_{\Xscr}\left[\x - \tfrac{1}{\beta} (\nabla_{{\x}} V_{\alpha}(\x)+{\bf e}_{\epsilon,M}) \right] \right\|^2 \\
             & \leq {2} \|\tilde{G}_{\beta, \epsilon,M}(\x)\|^2 + 2 \|{\bf e}_{\epsilon,M}\|^2.
        \end{align*}
		% \begin{align*}
        %     G_{\beta}(\x) & = \left\|  \beta \left(\x - \Pi_{\Xscr}\left[\x - \tfrac{1}{\beta} \nabla_{\x} {V_{\alpha}(\x)}\right]\right) \right\|^2 \\
        %      & = \left\| \beta\left( \x - \Pi_{\Xscr}\left[\x - \tfrac{1}{\beta} ({\nabla_{{\x}} V_{\alpha,\epsilon,{M}}(\x)}+{\bf e}_{\epsilon,{M}}) \right]\right) + \beta \Pi_{\Xscr}\left[\x - \tfrac{1}{\beta} \nabla_{{\x}} {V_{\alpha}(\x)}\right] \right.^2 \\
        %     & - \left. \beta\Pi_{\Xscr}\left[\x - \tfrac{1}{\beta} ({\nabla_{{\x}} V_{\alpha,\epsilon,{M}}(\x)}+{\bf e}_{\epsilon,{M}}) \right] \right\|^2 \\
        %      & \leq  2 \left\| \beta\left( \x - \Pi_{\Xscr}\left[\x - \tfrac{1}{\beta} ({\nabla_{{\x}} V_{\alpha,\epsilon,M}(\x)}+{\bf e}_{\epsilon,M}) \right]\right)\right\|^2 \\
        %      & + \left\|\beta \Pi_{\Xscr}\left[\x - \tfrac{1}{\beta} \nabla_{{\x}} {V_{\alpha}(\x)}\right] -\beta\Pi_{\Xscr}\left[\x - \tfrac{1}{\beta} (\nabla_{{\x}} V_{\alpha,\epsilon,M}(\x)+{\bf e}_{\epsilon,M}) \right] \right\|^2 \\
        %      & \leq \|\tilde{G}_{\beta, \epsilon,M}(\x)\|^2 + 2 \|{\bf e}_{\epsilon,M}\|^2 \, .
        % \end{align*}
    \end{proof}
\end{lemma}
    {This lays the foundation for analyzing} the sequences $\{\x_k\}$ and $G_{\beta}(\x_k)$ generated by the proposed method. We adopt a proof technique analogous to that employed in~\cite{beck14introduction}.   
\begin{lemma}\label{lem:inexact_proj_2}\em
    Suppose {Assumption}~\ref{ass:ass-1} holds. Let $\{\x_k \}$ be generated by Algorithm~\ref{algorithm:NI_inexact} with $\gamma_{k} < \frac{1}{L_{\uvs{1}}^{V_{\alpha}}}$ for $k \in \{1, \dots, K \}$. Then we have for all $k \in \{1, \dots, K\}$ 
    \begin{align} \label{eqn:residual_rel}
        V_{\alpha} \left( \x_{k+1} \right) \leq V_{\alpha} \left( \x_{k} \right) \uvs{- \left(1-{L_{1}^{V_{\alpha}}} \gamma_{k}\right)} \tfrac{\gamma_k}{4}\|G_{1/\gamma_0}{(\x_k)}\|^2 +\left(1 - \tfrac{{L_1^{V_{\alpha}}} \gamma_k}{2} \right) \gamma_k \|{\bf e}_{k,M_k}(\x_k)\|^2.
    \end{align}
\end{lemma}
\begin{proof} 
    Recall that {in a prior result ~\eqref{eqn:V_smooth}, we prove that} {$V_{\alpha}(\bullet)$} is {$L$-smooth} on $\Xscr$ with parameter ${L_{1}^{V_{\alpha}}}$. By the descent lemma for any $k > 0$, we have that
\begin{align*}
    V_{\alpha} \left( \x_{k+1} \right) & \leq V_{\alpha} \left( \x_{k} \right)+ \nabla_{\x} V_{\alpha}(\x_k)^{\top} \left( \x_{k+1} - \x_{k} \right) + \tfrac{{L_{1}^{V_{\alpha}}}}{2} \|\x_{k+1} - \x_{k}\|^{2} \\
    & = V_{\alpha} \left( \x_{k} \right)+  \left( \nabla_x V_{\alpha}(\x_k)+{\bf e}_{k,M_k}\right)^{\top} \left( \x_{k+1} - \x_{k} \right)- {{\bf e}_{\epsilon_k,M_k}(\x_k)}^{\top} \left( \x_{k+1} - \x_{k} \right) \\
    & + \tfrac{{L_{1}^{V_{\alpha}}}}{2} \|\x_{k+1} - \x_{k}\|^{2}.  
\end{align*}
Invoking the properties of the Euclidean projection operator, we have that
\begin{align*}
    & \left(\x_{k} - \gamma_k \left(\nabla_{x}V_{\alpha}(\x_k) + {{\bf e}_{\epsilon_k,M_k}(\x_k)} \right) - \x_{k+1}\right)^{\top} \left(\x_k - \x_{k+1} \right) \leq 0 \\
&\implies \left( \nabla_{x}V_{\alpha}(\x_k) + {{\bf e}_{\epsilon_k,M_k}(\x_k)} \right)^{\top} \left( \x_{k+1} - \x_{k} \right) \leq -\tfrac{1}{ \gamma_{k}}\|\x_{k+1} - \x_{k}\|^2. 
\end{align*}
Additionally note that for any two vectors $\uu, \vv \in \mathbb{R}^n$, \fyy{it} holds that 
\begin{align*}
\left(\uu^{\top} \vv \right)= \left(\gamma_{k}^{1/2} \uu \right)^{\top}\left(\gamma_{k}^{-1/2}\vv \right) \leq\frac{1}{2} \left( \gamma_{k} \| \uu \|^2 + \frac{1}{\gamma_k} \| \vv \|^2 \right).
\end{align*}
We conclude that 
\begin{align*}
    -{{\bf e}_{\epsilon_k,M_k}(\x_k)}^{\top} \left( \x_{k+1} - \x_{k} \right) \leq \tfrac{\gamma_k}{2} \|{\bf e}_{\epsilon_k,M_k}(\x_k)\|^2 +  \tfrac{1}{2 \gamma_k}\|\x_{k+1} - \x_{k}\|^2.
\end{align*}
As a consequence of the three preceding inequalities we may conclude that 
\begin{align*}
    V_{\alpha} \left( \x_{k+1} \right) & \leq V_{\alpha} \left( \x_{k} \right) -\tfrac{1}{\gamma_k}\|\x_{k+1} - \x_{k}\|^2 + \tfrac{\gamma_k}{2 }\|{\bf e}_{\epsilon_k,M_k}(\x_k)\|^2 +\tfrac{1}{2 \gamma_k}\|\x_{k+1} - \x_{k}\|^2  +\tfrac{L_1^{V_{\alpha}}}{2}\|\x_{k+1} - \x_{k}\|^2   \\
    & =  V_{\alpha} \left( \x_{k} \right) + \left(-\tfrac{1}{2\gamma_k} + {\tfrac{L_1^{{V}_\alpha}}{2 }}\right)\|\x_{k+1} - \x_{k}\|^2 + \tfrac{\gamma_k}{2} \|{{\bf e}_{\epsilon_k,M_k}(\x_k)}\|^2 .
\end{align*}
Recall that $\gamma_k  < \frac{1}{{L_{1}^{V_{\alpha}}}}$  by assumption. This allows us to write 
\begin{align*}
    V_{\alpha} \left( \x_{k+1} \right) & \leq V_{\alpha} \left( \x_{k} \right) + \left( \tfrac{{L_{1}^{V_{\alpha}}}}{2}-\tfrac{1}{2\gamma_k}\right)\|\x_{k+1} - \x_{k}\|^2 + \tfrac{\gamma_k}{2 }\|{{\bf e}_{\epsilon_k,M_k}(\x_k)}\|^2 \\
    & = V_{\alpha} \left( \x_{k} \right) + \left( \tfrac{{L_1^{V_{\alpha}}}}{2} -\tfrac{1}{2\gamma_k}\right) {\gamma_k^2}\|\tilde{G}_{{1/\gamma_{k}},{\epsilon_k}, {M_k}}(\x_k)\|^2 + \tfrac{\gamma_k}{2 }\|{{\bf e}_{\epsilon_k,M_k}(\x_k)}\|^2 \\
    & = V_{\alpha} \left( \x_{k} \right) + \left( {L_{1}^{V_{\alpha}}} \gamma_k - 1 \right) \tfrac{\gamma_k}{2}\|\tilde{G}_{{1/\gamma_{k}},{\epsilon_k}, {M_k}}(\x_k)\|^2 + \tfrac{\gamma_k}{2 }\|{{\bf e}_{\epsilon_k,M_k}(\x_k)}\|^2 \\
    & \overset{\tiny \mbox{{Lemma}}~\ref{bd_G}}{\leq} V_{\alpha} \left( \x_{k} \right) + \left({L_{1}^{V_{\alpha}}} \gamma_k  -1 \right) \tfrac{\gamma_k}{4}\|G_{{1/\gamma_{k}}}(\x_k)\|^2 \\
	& + \left( 1-L_{1}^{V_{\alpha}} \gamma_k \right) \tfrac{\gamma_k}{2} \|{{\bf e}_{\epsilon_k,M_k}(\x_k)}\|^2 + \tfrac{\gamma_k}{2 }\|{{\bf e}_{\epsilon_k,M_k}(\x_k)}\|^2 \\
    &  \overset{\tiny ~\cite[\mbox{Ch. 9, Lem.9.12}]{beck14introduction}}{\leq} V_{\alpha} \left( \x_{k} \right) \uvs{- \left( 1-{L_{1}^{V_{\alpha}}} \gamma_k \right)} \tfrac{\gamma_k}{4}\|G_{1/\gamma_{0}}(\x_k)\|^2 +\left(1 - \tfrac{{L_{1}^{V_{\alpha}}} \gamma_k}{2} \right) \gamma_k \|{{\bf e}_{\epsilon_k,M_k}(\x_k)}\|^2,
\end{align*}
which is the desired result. 
\end{proof}
%\begin{comment}
	We now present an almost sure convergence guarantee for the sequence generated by Algorithm~\ref{algorithm:NI_inexact}.
	\begin{comment}
	\begin{assumption}[{\bf \lm{Summability of the inexactness sequence}}]
	\label{ass:summable_inexactness}\em
	\lm{The inexactness sequence $\{\epsilon_k\}$ satisfies $\sum_{k=1}^{\infty} \epsilon_k < \infty$.}
	\end{assumption}
	\begin{assumption}[{\bf \lm{Summability of the sample size sequence}}]
	\label{ass:summable_sample_size}\em
	The sample size sequence $\{M_k\}$ satisfies $\sum_{k=1}^{\infty} \frac{1}{M_k} < \infty$.
	\end{assumption}
	\end{comment}
\begin{proposition}[Asymptotic guarantees for Alg.~\ref{algorithm:NI_inexact}]\em
    Let Assumption~\ref{ass:ass-1} holds. Let $\{\x_k\}$ be generated by Algorithm \ref{algorithm:NI_inexact}. Let $\{\gamma_k\}$ be a non-increasing sequence of stepsizes, \uvs{where $\gamma_0 < \frac{1}{L_1^{V_{\alpha}}}$ and} $\gamma_k > 0$ for all $k.$ Assume that 
      $\sum_{k=1}^{\infty} \mathbb{E}\left[\|{{\bf e}_{\epsilon_k,M_k}(\x_k)}\|^2 \mid \x_k \right] = \sum_{k=1}^{\infty} \left( \tfrac{(3N+2)N\sigma^2}{M_k} +{ \left( \left({2N+2} \right)\sum_{\nu =1}^{N} (L_{1}^{\nu})^2 + 4 \alpha \right) \epsilon_k}\right) <\infty$ almost surely.  Then  
 $\|G_{1/\gamma_{{0}}}(\x_k)\| \xrightarrow[k \to \infty]{a.s.} 0.$
\label{thm:almost_sure}
\end{proposition}
\begin{proof}
First, recall from Proposition~\ref{thm:v_alpha_properties} that $V_{\alpha}^* = \inf_{\x \in \Xscr} V_{\alpha}(\x) = 0.$  By taking conditional expectations with respect to $\x_k$ on the both sides of the inequality \eqref{lem:inexact_proj_2}, we have 
\begin{align}
\mathbb{E}\left[V_{\alpha}\left( \x_{k+1} \right) - V_{\alpha}^*  \mid \x_k\right] & \leq V_{\alpha} \left( \x_{k} \right) - \left(1 - {L_{1}^{V_{\alpha}}} \gamma_k\right) \tfrac{\gamma_k}{4}\|G_{1/\gamma_{0}}(\x_k)\|^2 \\
&  +\left(1 - \tfrac{{L_{1}^{V_{\alpha}}} \gamma_k}{2} \right) \gamma_k \mathbb{E}\left[\|{{\bf e}_{\epsilon_k,M_k}(\x_k)}\|^2 \mid \x_k \right].
    \end{align} 
\uvs{Since $\left(1 - {L_{1}^{V_{\alpha}}} \gamma_k\right) > 0$ for any $k$ by choice of $\gamma_k$}, the assumed sumability of
    $\sum_{k=1}^{\infty} \mathbb{E}\left[\|{{\bf e}_{\epsilon_k,M_k}(\x_k)}\|^2 \mid \x_k \right]$ ({which holds by} our additional assumption), and the nonnegativity of
    $V_{\alpha}(\x_{k}) - V^*_\alpha = V_{\alpha}(\x_{k})$, we have \uvs{by invoking the Robbins-Siegmund Lemma} that
    $\{V_{\alpha}(\x_{k}) -V^*_\alpha\}$ is convergent almost surely and
    $$\sum_{k=1}^{\infty} \left\|{G}_{1/\gamma_{{0}}}(\x_k)\right\|^2 < \infty$$
almost surely. It remains to show that with probability one, $\|G_{1/\gamma_{{0}}}(\x_k)\| \to 0$ as $k \to \infty$.
We proceed  with a proof by contradiction. Suppose for $\omega \in \Omega_1 \subset
\Omega$ and $\mu(\Omega_1) > 0$ (\fyy{i.e.,} with finite probability),
$\|G_{1/\gamma_{{0}}}(\x_k)\| \xrightarrow{k \in \Kscr(\omega)} \epsilon(\omega)
> 0$ where $\Kscr(\omega)$ is a random subsequence. Consequently, for every
$\omega \in \Omega_1$ and $\tilde{\varepsilon} > 0$, there exists $K(\omega)$ such
that $k \geq K(\omega)$, $\| G_{1/\gamma_{{0}}}(\x_k) \| \geq
\tfrac{\epsilon(\omega)}{2}$. Consequently, we have that $\sum_{k \to
\infty} \|G_{1/\gamma_{{0}}}(\x_k)\|^2 \geq \sum_{k \in \Kscr(\omega)}
\|G_{1/\gamma_{{0}}} (\x_k)\|^2 \geq \sum_{k \in \Kscr(\omega), k \geq
K(\omega)} \|G_{1/\gamma_{{0}}} (\x_k)\|^2 = \infty$ with finite probability.
But this leads to a contradiction, implying that $\|G_{1/\gamma_{{0}}}(\x_k)\|^2
\xrightarrow[k \to \infty]{a.s.} 0.$
\end{proof}
\smallskip
    We have proven almost sure convergence of Algorithm \ref{algorithm:NI_inexact} provided that one may access $\epsilon-$approximate solutions and batch size sequences such that $\sum_{k=1}^{\infty} \mathbb{E}\left[\|{\bf e}_{\epsilon_k,M_k}(\x_k)\|^2 \mid \x_k \right] < \infty$ almost surely. %We provide an example of one such algorithm in section \ref{sec:inexact_smoothed_gnep_algorithm}.
\smallskip
Next, we present the main rate and complexity result for our stated inexact scheme for computing Nash equilibria. We present the rates in terms of $\epsilon_k$, so that they are more general and may be interpreted for schemes that achieve different levels of asymptotic accuracy.  Before proceeding, we prove the following lemma which we employ to produce rate statements {an iterate with randomly specified indexing} for different choices of stepsize sequences $\{\gamma_k\}.$ \\
\begin{lemma}\label{lemma:bound_res}\em
	Let Assumptions \ref{ass:ass-1} hold. Let $\{\x_k\}$ be generated by Algorithm~\ref{algorithm:NI_inexact}. Let $R_{\ell, K}$ be a random integer on $\{ \lceil\lambda K \rceil := \ell, \dots, K-1 \}$ for some $\lambda \in [0.5,1)$, {with probability mass function given by
		\begin{align}\label{eqn:p_r_def}
		\mathbb{P}_{{R_{\ell,K}}}(R_{\ell,K} = j) = {\tfrac{\gamma_j}{\sum_{i=\ell}^{K-1}   \gamma_i}}
		\end{align}
		for all $\ell\leq j\leq K-1$.} Let $\{\gamma_k\}$ be a non-increasing sequence (\fyy{i.e.,} constant or diminishing) such that $\gamma_0 < \uvs{\frac{1}{{L_{1}^{V_{\alpha}}}}}$ and let the batchsize sequence $\{M_k\}$ be non-decreasing (\fyy{i.e.,} constant or increasing). 
%We emphasize $R_{\ell,K}$ depends on both $K$ and $\ell$ (or equivalently, $K$ and $\ell$). 
Then the following inequality holds.
	 \begin{align}
		\label{eqn:general_bound}
		 \mathbb{E}\left[\|G_{1/\gamma_0}(\x_{R_{\ell,K}})\|{^2}\right]   & \leq   \tfrac{ 4 \sum_{k=\ell}^{{K-1}} \gamma_k \left(\tfrac{20N\sigma^2}{M_{{k}}} +  \left( \left({2N+2} \right)\sum_{\nu =1}^{N} (L_{1}^{\nu})^2 + 4 \alpha \right) \epsilon_k\right) + \mathbb{E}\left[V_{\alpha} \left( \x_{\ell} \right)\right]}{\left(1 - {L_{1}^{V_{\alpha}}} \gamma_0 \right) \left(  \sum_{k = \ell} ^{K-1} \gamma_k \right)}.
		\end{align}
\end{lemma}

	\begin{proof}
		\noindent  Consider the inequality \eqref{eqn:residual_rel}.
		\begin{align*}
			V_{\alpha} \left( \x_{k+1} \right) \leq V_{\alpha} \left( \x_{k} \right) \uvs{- \left(1-{L_{1}^{V_{\alpha}}} \gamma_k \right)} \tfrac{\gamma_k}{4}\|G_{1/\gamma_0}(\x_k)\|^2 +\left(1 - \tfrac{{L_{1}^{V_{\alpha}}} \gamma_k}{2} \right) \gamma_k \|{{\bf e}_{\epsilon_k,M_k}(\x_k)}\|^2.
		\end{align*}
		Rearranging, we have
		\begin{align*} \left(1- {L_{1}^{V_{\alpha}} \gamma_k }\right) \tfrac{\gamma}{4}\|G_{1/\gamma_k}(\x_{k})\|^2   \leq V_{\alpha} \left( \x_{k} \right) - V_{\alpha}\left( \x_{k+1} \right)   +\left(1 - \tfrac{{L_{1}^{V_{\alpha}}} \gamma_k}{2} \right) \gamma_k \|{{\bf e}_{\epsilon_k,M_k}(\x_k)}\|^2 \, .
		\end{align*}
		Summing from $k =\ell, \ldots, K-1$ where $\ell\triangleq  \lceil \lambda K\rceil$ we have
		\begin{align*} \sum_{k= \ell}^{K-1}  \left(1- {L_{1}^{V_{\alpha}} \gamma_k }\right) \tfrac{\gamma_k}{4} \|G_{1/\gamma_k}(\x_{k})\|^2   \leq V_{\alpha} \left( \x_{\ell} \right) - V_{\alpha}\left( \x_{K} \right)   +  \sum_{k= \ell}^{K-1} \left(1 - \tfrac{{L_{1}^{V_{\alpha}}} \gamma_k}{2} \right) \gamma_k \|{{\bf e}_{\epsilon_k,M_k}(\x_k)}\|^2.
		\end{align*}
		Taking expectations with respect to the iterate trajectory {on} both sides, and once more noting that $V_{\alpha}(\x_K) \geq V_{\alpha}^* = 0$, we obtain
		\begin{align*}
			\sum_{k= \ell} ^{K-1} \left(1 - {L_{1}^{V_{\alpha}}} \gamma_k \right) \tfrac{\gamma_k}{4} \mathbb{E}\left[\|G_{1/\gamma_k}(\x_{k})\|{^2}\right]   \leq  \sum_{k= \ell}^{K-1} \left(1 - \tfrac{{L_{1}^{V_{\alpha}}} \gamma_k}{2} \right) \gamma_k \mathbb{E}\left[\|{{\bf e}_{\epsilon_k,M_k}(\x_k)}\|^2\right] + \mathbb{E}\left[V_{\alpha} \left( \x_{\ell} \right)\right] - V_{\alpha}^*.
		\end{align*}
		 First, observe that for a non-increasing stepsize sequence 
		 $\{ \gamma_k\}$, $\|G_{1/\gamma_k}(\x_{k})\|^2 < \|G_{1/\gamma_{k+1}}({\x_{k}})\|^2 $ {(see~\cite{beck14introduction})}. With this in mind and by invoking the definition of the probability measure $\mathbb{P}_R$, note that
		\begin{align}
			\label{eqn:error_bound}
			\left(1 - {L_{1}^{V_{\alpha}}} \gamma_0 \right)  \left(\sum_{k = \ell} ^{K-1} \tfrac{\gamma_k}{4} \right)   \left( \mathbb{E} \left[\mathbb{E}_{R}\left[\|G_{1/\gamma_0}(\x_{R_{\ell,K}})\|{^2}\right] \right] \right)  & \leq \sum_{k= \ell} ^{K-1} \left(1 - {L_{1}^{V_{\alpha}}} \gamma_k \right) \tfrac{\gamma_k}{4} \mathbb{E}\left[\|G_{1/\gamma_k}(\x_{k})\|{^2}\right] \\
\notag & \leq   \sum_{k= \ell}^{K-1} \left(1 - \tfrac{{L_{1}^{V_{\alpha}}} \gamma_k}{2} \right) \gamma_k \mathbb{E}\left[\|{{\bf e}_{\epsilon_k,M_k}(\x_k)}\|^2\right] \\
\notag		& + \mathbb{E}\left[V_{\alpha} \left( \x_{\ell} \right)\right] - V_{\alpha}^*.
		\end{align}
	By the fact that $\mathbb{E}\left[\|{{\bf e}_{\epsilon_k,M_k}(\x_k)}\|^2 \mid \x_k\right] \leq   {\tfrac{{(3N + 2)N\sigma}^2}{M_k}  + \left( \left({2N+2} \right)\sum_{\nu =1}^{N} {(L_{1}^{\nu})^2} + 4 \alpha \right) \epsilon_k }$ and by observing that $\left(1 - \tfrac{{L_{1}^{V_{\alpha}}} \gamma_k}{2} \right) < 1$, we obtain the desired result.
		\begin{align*}
			\mathbb{E}\left[\|G_{1/\gamma_0}(\x_{R_{\ell,K}})\|{^2}\right]   & \leq   \tfrac{4\sum_{k=\ell}^{{K-1}} \gamma_k \left({\frac{{(3N + 2)N\sigma}^2}{M_k}  + \left( \left({2N+2} \right)\sum_{\nu =1}^{N} {(L_{1}^{\nu})^2} + 4 \alpha \right) \epsilon_k }\right) + \mathbb{E}\left[V_{\alpha} \left( \x_{\ell} \right)\right]}{\left(1 - {L_{1}^{V_{\alpha}}} \gamma_0 \right) \left(  \sum_{k = \ell} ^{K-1} \gamma_k \right)}. 
		\end{align*}
	\end{proof}
 A few remarks about this error bound bear mentioning. First, in the
above claim, $\mathbb{E}\left[\bullet\right]$ represents an expectation over
the set of replications and $\fyy{{R_{\ell,K}}}$. Note that it is not a guarantee
that the last iterate $\x_K$ will satisfy any particular bound but instead
provides a guarantee on $\x_{R_{\ell,K}}$.  With this error-bound in hand,
we present rate statements and provide complexity guarantees for a variety of
step size, mini-batch, and inexactness sequences. 
\begin{theorem}[{\bf Convergence rate: Constant {stepsize} rule}]\label{thm:NI_inexact_rate_constant}\em
    Let Assumption \ref{ass:ass-1} hold. Let $\{\x_k\}$ be generated by {Algorithm} \ref{algorithm:NI_inexact}.
	Let $\gamma_k =\uvs{\gamma_0}$ \uvs{for every $k \ge 0$}, where $\uvs{\gamma_0} \, < \, \uvs{\frac{1}{ L_{{1}}^{V_{\alpha}}}}$. {Let $R_{\ell, K}$ be a random integer on $\{ \lceil\lambda K \rceil := \ell, \dots, K-1 \}$ for some $\lambda \in [0.5,1)$, where $K> \tfrac{2}{1-\lambda}$}, $\ell\triangleq  \lceil \lambda K\rceil$, and  $M_k := \lceil 1 + a k  \rceil$ for all $k$  for some user-defined sampling growth rate $a > 0 $. 

     \noindent {\bf{(i)}} {Then the following holds for any $K,\ell,\gamma_0$ as prescribed above.}
    \begin{align*}
	  	\mathbb{E}\left[\|G_{1/\gamma_0}(\x_{R_{\ell,K}})\|{^2}\right]   & \leq \tfrac{8}{K-\ell}  \left(\tfrac{{(1-\ln(\lambda))(3N + 2)N\sigma^2}}{a}  +  { \left(\left({2N+2} \right)\sum_{\nu =1}^{N} {(L_{1}^{\nu})^2} + 4 \alpha \right) \sum_{k=\ell}^{{K-1}}  \epsilon_k} + \tfrac{2 \mathbb{E}\left[V_{\alpha} \left( \x_{\ell} \right)\right] }{{L_{1}^{V_{\alpha}}}} \right).
\end{align*}
\noindent {\bf{(ii)}} {Suppose $\uvs{\gamma_0 = \frac{1}{ 2L_{{1}}^{V_{\alpha}}}}$ and $\epsilon_k = \tfrac{p}{k}$ where $p > 0$.} 
{Let $\varepsilon > 0 $ and an $\varepsilon$-solution satisfies $\mathbb{E}\left[\|G_{1/\gamma_0}(\x_{R_{\ell,K}})\|\right] < \varepsilon$ for some $K_{\varepsilon} > 0$.} Then {the iteration and sample-complexity for computing an $\varepsilon$-solution are $\mathcal{O}(\varepsilon^{-2})$ and $\mathcal{O}(\varepsilon^{-4})$, respectively.} % number of projection steps  is $K_{\varepsilon} = \mathcal{O}\left( \varepsilon^{-2} \ln(\varepsilon^{-1}) \right)$ and the total sample complexity is $\sum_{k=1}^{K_{\varepsilon}} {M_k}= \mathcal{O}\left( {\varepsilon}^{-4} {\ln^{2}(\varepsilon^{-1})} \right).$\\
%\noindent {\bf{(ii-2)}} Assume {that} the inexactness sequence $\{\epsilon_k\}$ diminishes at a rate so that $\sum_{k = \ell}^\infty \epsilon_k < C$ for some finite $C$. Let $\varepsilon > 0 $ and let $K_{\varepsilon}$ be such that $\mathbb{E}\left[\|G_{1/\gamma_0}(\x_{R_{\ell,K}})\|\right] < \varepsilon.$ Then total number of projection steps is $K_{\varepsilon} = \mathcal{O}\left( \varepsilon^{-2}  \right)$ and the total sample complexity is $\sum_{k=1}^{K_{\varepsilon}} {M_k} =\mathcal{O}\left( {\varepsilon}^{-4} \right).$\\
%\noindent {\bf{(ii-1)}} The total number of projection steps is $K_{\varepsilon} = \mathcal{O}\left( \varepsilon^{ -2} \right).$ \\
%\noindent {\bf{(ii-2)}} The total sample complexity is $\sum_{k=1}^{K_{\varepsilon}} = \mathcal{O}\left( \epsilon^{-4} \right).$
%\noindent {\bf{(ii-1)}}  The total number of projection steps is $K_{\varepsilon}$ and iteration complexities $\sum_{k=0}^{K_{\varepsilon}} M_k$ are as described below, given a inexactness sequence $\{\epsilon_k\}$ satisfying the condition listed in the first column. \\
 %\noindent {\bf{(ii-1)}} Consider the stepsize sequence $\gamma_k := \frac{\gamma_0}{\sqrt{k+1}} $, with $\gamma_0 < \frac{2}{ \lm{L_{0}^{V_{\alpha}}}}$. Choose $\lambda \in (0.5 ,1) $, $K> \tfrac{2}{1-\lambda}$ with $\ell\triangleq  \lceil \lambda K\rceil$. Then the expectation of the square of the residual at iteration $R$ may be bounded as written below:  
 \end{theorem}
\begin{proof}
\noindent {\bf (i)} Consider the inequality \eqref{eqn:error_bound}, restated here for clarity:
\begin{align*}
	\mathbb{E}\left[\|G_{1/\gamma_0}(\x_{R_{\ell,K}})\|{^2}\right]   & \leq   \tfrac{4\sum_{k=\ell}^{K} \gamma_k \left({\tfrac{{(3N + 2)N\sigma}^2}{M_k}  + \left( \left({2N+2} \right)\sum_{\nu =1}^{N} {(L_{1}^{\nu})^2} + 4 \alpha \right) \epsilon_k }\right) + \mathbb{E}\left[V_{\alpha} \left( \x_{\ell} \right)\right]}{\left(1 - {L_{1}^{V_{\alpha}}} \gamma_0 \right) \left(  \sum_{k = \ell} ^{K-1} \gamma_k \right)}. 
\end{align*}
To simplify the notation, we denote the
quantity associated with the variance of the
estimator of $\nabla_{\x}V_{\alpha}(\x)$ by
{defining $\rho$ and $\mu$ such that $\rho
= {(3N + 2)N\sigma}^2$ and}
$\mu := \left({2N+2} \right)\sum_{\nu =1}^{N}
{(L_{1}^{\nu})^2} + 4 \alpha $. With this
notation in hand, our bound {can be more simply stated as} 
  \begin{align*}
  	\mathbb{E}\left[\|G_{1/\gamma_0}(\x_{R_{\ell,K}})\|{^2}\right]   & \leq   \tfrac{4\sum_{k=\ell}^{{K-1}} \gamma_k \left(\frac{\rho}{M_k}  + \mu \epsilon_k \right) + \mathbb{E}\left[V_{\alpha} \left( \x_{\ell} \right)\right]}{\left(1 - {L_{1}^{V_{\alpha}}} \gamma_0 \right) \left(  \sum_{k = \ell} ^{K-1} \gamma_k \right)}. 
  \end{align*}
  Note that $K > \frac{2}{(1- \lambda)}$ implies that  
  \begin{align*}
  \uvs{\sum_{k=\ell}^{K-1} \tfrac{1}{\uvs{M_k}} \le}	\sum_{k=\ell}^{K-1} \tfrac{1}{ak+1}  = \sum_{k=\ell}^{K-1}  \tfrac{1}{a(k+\frac{1}{a})}  & = a^{-1} \sum_{k=\ell}^{K-1} \tfrac{1}{(k+\frac{1}{a})}\\
  	& \leq a^{-1} \sum_{k=\ell}^{K-1} \tfrac{1}{k} \leq a^{-1}\left( \tfrac{1}{\ell} + \ln \left( K \right) - \ln \left(\ell \right) \right) \\
  	&  \leq a^{-1}\left(1 + \ln \left(\tfrac{K}{\lceil \lambda K \rceil} \right)\right) \leq a^{-1}(1 - \ln(\lambda)).
  \end{align*}
 Using this fact {and observing that $\sum_{k=\ell}^{{K-1}} \gamma_k \ge (K-\ell)\gamma_0$}, we have 
  \begin{align}
  	\notag \mathbb{E}\left[\|G_{1/\gamma_0}(\x_{R_{\ell,K}})\|{^2}\right]   & \leq   \tfrac{4a^{-1}(1-\ln(\lambda))\gamma_0 \rho}{\left(1 - {L_{1}^{V_{\alpha}}} \gamma_0 \right)(K - \ell) \gamma_0 }  +  \tfrac{4 \gamma_0 \mu \sum_{k=\ell}^{{K-1}}  \epsilon_k}{\left(1 - {L_{1}^{V_{\alpha}}} \gamma_0 \right)(K - \ell) \gamma_0 } + \tfrac{4 \mathbb{E}\left[V_{\alpha} \left( \x_{\ell} \right)\right] }{\left(1 - {L_{1}^{V_{\alpha}}} \gamma_0 \right)(K - \ell) \gamma_0 } \\   & \leq   \tfrac{4a^{-1}(1-\ln(\lambda))\rho}{\left(1 - {L_{1}^{V_{\alpha}}} \gamma_0 \right)(K - \ell)  }  +  \tfrac{4  \mu \sum_{k=\ell}^{{K-1}}  \epsilon_k}{\left(1 - {L_{1}^{V_{\alpha}}} \gamma_0 \right)(K - \ell)} + \tfrac{4 \mathbb{E}\left[V_{\alpha} \left( \x_{\ell} \right)\right] }{\left(1 - {L_{1}^{V_{\alpha}}} \gamma_0 \right)(K - \ell) \gamma_0 } \label{bd:three},
\end{align}
which, upon substituting for $\rho$ and $\mu$,  is the desired result. 

\noindent {\bf (ii)} From the relationship in part {\bf(i)} and by \uvs{noting that}  $\gamma_k = \gamma_0 := \frac{1}{2 {L_{1}^{V_{\alpha}}}}$ \uvs{for any $k \ge 0$}, we obtain
\begin{align*}
	\mathbb{E}\left[\|G_{1/\gamma_0}(\x_{R_{\ell,K}})\|{^2}\right]   & \leq  \tfrac{8 \mu  \sum_{k=\ell}^{{K-1}}\epsilon_k}{(K - \ell)  } +  \tfrac{8 a^{-1}(1-\ln(\lambda))\rho}{(K-\ell)}  +  \tfrac{16 {L_{1}^{V_{\alpha}}} \mathbb{E}\left[V_{\alpha} \left( \x_{\ell} \right)\right] }{(K - \ell) } \\
	& = \mathcal{O}\left( \tfrac{1}{K} + \tfrac{\sum_{k = \ell}^K \epsilon_{k}}{K} \right) = \mathcal{O}\left( \tfrac{1}{K}\right),
\end{align*}
 %First, assume that  $\sum_{k=0}^{K} \epsilon_k \leq \mathcal{O}\left(\sqrt{K}\right).$ Then 
%\begin{align*}
%	\mathbb{E}\left[\|G_{1/\gamma_0}(\x_{R_{\ell,K}})\|\right]^2 	& \leq \sqrt{\mathcal{O}\left( \frac{1}{K} + \frac{\sum_{k=0}^{K} \epsilon_k}{K} \right)} \\
%	& = \sqrt{\mathcal{O}\left( \frac{1}{K} + \frac{1}{\sqrt{K}} \right)} \, .
%\end{align*}
%So $K_{\varepsilon} = \mathcal{O}\left( \varepsilon^{-4} \right)$. \\
where $\sum_{k=\ell}^{{K-1}} \epsilon_k \le \lm{p}(1-\ln(\lambda))$ from our prior discussion. Consequently, by Jensen's inequality,  
%By assumption,  $\sum_{k=0}^{K} \epsilon_k \leq \mathcal{O}\left(\ln \left(K\right)\right).$ We therefore have
\begin{align*}
	\mathbb{E}\left[\|G_{1/\gamma_0}(\x_{R_{\ell,K}})\|\right] 	& \leq {\sqrt{\mathcal{O}\left( \tfrac{1}{K} \right)}} = {\mathcal{O}\left( \tfrac{1}{\sqrt{K}} \right)}.
\end{align*}
Rearranging, for $K_{\varepsilon}$, we have that for some sufficiently large constant $Q > 0$,
\begin{align*}
	\mathbb{E}\left[\|G_{1/\gamma_0}(\x_{R_{\ell,K}})\|\right] & \leq Q \left( \tfrac{1}{{\sqrt{K_{\varepsilon}}}} \right)  \leq {\varepsilon}.
\end{align*}
{This implies that $K_{\varepsilon} = \mathcal{O}(\varepsilon^{-2})$ and the sample-complexity is bounded as} 
%Consider $\tilde{K_{\varepsilon}} :=  5(1+\delta)\varepsilon^{-2} \ln \left( \varepsilon^{-2}\right)$ for some $\delta >0$. Then
%{\begin{align*} \tfrac{\ln \left( \tilde{K_{\varepsilon}} \right) }{\tilde{K_{\varepsilon}}}  & =  \tfrac{\ln \left(5(1 +\delta)  \varepsilon^{-2} \ln \left( \varepsilon^{-2}\right) \right)}{ 5(1+ \delta) \varepsilon^{-2} \ln \left(  \varepsilon^{-2}\right)}  \\ & = \left(\varepsilon^{2}\right) \tfrac{\ln(5) + \ln \left((1 + \delta)\varepsilon^{-2} \right) + \ln \left( \ln \left((1 + \delta) \varepsilon^{-2}\right) \right)}{ (1 + \delta)  \ln \left( \varepsilon^{-2}\right)} \\ & = \left(\varepsilon^{2}\right) \left(\tfrac{\ln(5) +  \ln (1 + \delta) + \ln(\varepsilon^{-2})}{ 5(1 + \delta)  \ln \left( \varepsilon^{-2}\right)} + \tfrac{\ln \left( \ln \left((1 + \delta) \varepsilon^{-2}\right) \right)}{ 5(1 + \delta)  \ln \left( \varepsilon^{-2}\right)} \right)\\ & \le \left(\tfrac{\varepsilon^{2}}{5}\right) \left(1 + \tfrac{1}{  \ln \left( \varepsilon^{-2}\right)} + \tfrac{1}{ (1 + \delta)} + \tfrac{1}{  \ln \left( \varepsilon^{-2}\right)}+ \tfrac{1}{ (1 + \delta)}\right) \leq  \varepsilon^2.  \end{align*}}
%{Observing that $\ln \left( \varepsilon^{-2}\right) = 2 \ln \left(\varepsilon^{-1}\right)$}, {we conclude that} $K_{\varepsilon} = \mathcal{O}\left( \varepsilon^{-2} \ln \left( \varepsilon^{-1}\right) \right)$. \\
\begin{align*}
	\sum_{k=0}^{{K_{\varepsilon}-1}} M_k = \sum_{k=0}^{{K_{\varepsilon}-1}} \lceil 1 + ak \rceil \leq \mathcal{O} \left(K_{\varepsilon}\right) + \mathcal{O} \left(a {K_{\varepsilon}^2}\right)  = \mathcal{O} \left({\varepsilon^{-4}}\right). 
%\left(\varepsilon^{-4}{\ln^2}(\varepsilon^{-1})\right).
\end{align*}
%\noindent {\bf (ii-2)} Now, assume that  $\sum_{k=0}^{K} \epsilon_k \leq C,$ for some constant $C < \infty.$ Then, 
% \begin{align*}
 	%\mathbb{E}\left[\|G_{1/\gamma_0}(\x_{R_{\ell,K}})\|\right] 	& \leq  \sqrt{\mathcal{O}\left( \tfrac{1}{K} + \tfrac{C}{{K}} \right)}  \leq \sqrt{\mathcal{O}\left( \tfrac{1}{K} \right)}  \, .
% 	\mathbb{E}\left[\|G_{1/\gamma_0}(\x_{R_{\ell,K}})\|\right] 	& \leq \sqrt{\mathcal{O}\left( \tfrac{1}{K} \right)}  \, .
% \end{align*}
%Therefore, $K_{\varepsilon} = \mathcal{O}\left( \varepsilon^{-2} \right)$. Note that this case corresponds to the assumption placed on the inexactness sequence $\{\epsilon_k\}$ in statement \eqref{thm:almost_sure}. The sample-complexity is 
%\begin{align*}
%	\sum_{k=0}^{K_{\varepsilon}} M_k & = \sum_{k=0}^{K_{\varepsilon}} \lceil 1 + a(k) \rceil  \leq \mathcal{O} \left(K_{\varepsilon}\right) + \mathcal{O} \left(a {K_{\varepsilon}^2}\right) = \mathcal{O} \left(\varepsilon^{-4}\right). 
%\end{align*}
\end{proof}

Observe {from \eqref{bd:three}} that the
bound on the residual consists of three
terms. The first term is a result of the
variance associated with the mini-batch
gradient estimate, and diminishes as the
{batch} size increases. The second term is a
result of the inexactness in the resolution
of the subproblem, and accordingly diminishes
with the level of inexactness.  The third
term is a results of the ``bias'' of the
first iteration, diminishing with the number
of iterations.  We now present {analogous
guarantees} when the stepsize sequence is
diminishing.
\begin{proposition}[{\bf Convergence rate: Diminishing stepsize rule}]\label{thm:NI_inexact_rate_diminishing}\em
    Let Assumption \ref{ass:ass-1} hold. Let $\{\x_k\}$ be generated by {Algorithm} \ref{algorithm:NI_inexact}.
	Let $\gamma_k =\tfrac{\gamma_0}{\sqrt{k}} $, where $\gamma_0<\uvs{\frac{1}{ L_{{1}}^{V_{\alpha}}}}$. {Let $R_{\ell, K}$ be a random integer on $\{ \lceil\lambda K \rceil := \ell, \dots, K-1 \}$ for some $\lambda \in [0.5,1)$, where $K> \tfrac{2}{1-\lambda}$}, $\ell\triangleq  \lceil \lambda K\rceil$, and  $M_k := \lceil 1 + a {\sqrt{k}}  \rceil$ for all $k$  for some user-defined sampling growth rate $a > 0 $. 

     \noindent {\bf{(i)}} {Then the following holds for any $K,\ell,\gamma_0$ as prescribed above.}
\begin{align*}
			\mathbb{E}\left[\|G_{1/\gamma_0}(\x_{R_{\ell,K}})\|{^2}\right]   & \leq \uvs{\tfrac{1}{ \sqrt{K}}} \left( {\uvs{8}(1-\ln(\lambda))\rho}  +  { \uvs{8}\mu \sum_{k = \ell}^{{K-1}} \gamma_k \epsilon_{k}} + \tfrac{2 \mathbb{E}\left[V_{\alpha} \left( \x_{\ell} \right)\right] }{\uvs{\gamma_0}} \right).
\end{align*}
\noindent {\bf{(ii)}} {Suppose $\uvs{\gamma_0 = \tfrac{1}{2L^{V_{\alpha}}_1}}$ and $\epsilon_k = \tfrac{p}{\sqrt{k}}$ where $p > 0$.} 
%Assume that the inexactness sequence $\{\epsilon_k\}$ diminishes at a rate of $\mathcal{O}\left(\frac{1}{K}\right).$ 
{Let $\varepsilon > 0 $ and an $\varepsilon$-solution satisfies $\mathbb{E}\left[\|G_{1/\gamma_0}(\x_{R_{\ell,K}})\|\right] < \varepsilon$ for some $K_{\varepsilon} > 0$.} Then {the iteration and sample-complexity for computing an $\varepsilon$-solution are $\mathcal{O}(\varepsilon^{-4})$ and $\mathcal{O}({\varepsilon^{-6}})$, respectively.} 
 \end{proposition}
	\begin{proof}
\noindent {\bf (i)} Consider the inequality \eqref{eqn:error_bound}, restated here for clarity, and let us denote  
$\rho := {(3N + 2)N\sigma}^2$ and $\mu := \left({2N+2} \right)\sum_{\nu =1}^{N} {(L_{1}^{\nu})^2} + 4 \alpha $. We have
\begin{align*}
  	\mathbb{E}\left[\|G_{1/\gamma_0}(\x_{R_{\ell,K}})\|{^2}\right]   & \leq   \tfrac{4\sum_{k=\ell}^{{K-1}} \gamma_k \left(\tfrac{\rho}{M_k}  + \mu \epsilon_k \right) + \mathbb{E}\left[V_{\alpha} \left( \x_{\ell} \right)\right]}{\left(1 - {L_{1}^{V_{\alpha}}} \gamma_0 \right) \left(  \sum_{k = \ell} ^{K-1} \gamma_k \right)}.  
\end{align*}
We need to provide upper bounds on $\sum_{k=\ell}^{{K-1}}  \tfrac{\gamma_k}{M_{k}}$ and $\sum_{k=\ell}^{{K-1}} \gamma_k \epsilon_k. $ First, note that $\sum_{k=\ell}^{{K-1}}  \tfrac{\gamma_k}{M_{k}} < {\tfrac{\gamma_0}{a}(1-\ln(\lambda))}$.
Note that $K > \tfrac{2}{(1- \lambda)}$ implies that 
$\sum_{k=\ell}^{K-1} \tfrac{1}{k+1} \leq 1 - \ln(\lambda).$ Using this fact, and observing that $\sum_{k=\ell}^{K-1} \gamma_k \geq \int_{\ell-1}^{K-1} \tfrac{\gamma_0}{\sqrt{x+1}} dx \geq 2 \gamma_0 (1 - \sqrt{\lambda}) \sqrt{K},$ we have 
{
	\begin{align*}
		\mathbb{E}\left[\|G_{1/\gamma_0}(\x_{R_{\ell,K}})\|{^2}\right]   & \leq   \tfrac{4\uvs{\gamma_0}(1-\ln(\lambda))\rho}{\gamma_0 (1 - \sqrt{\lambda}) \sqrt{K} }  +  \tfrac{4 \uvs{\gamma_0}\mu \sum_{k = \ell}^K \gamma_k \epsilon_{k}}{ \gamma_0 (1 - \sqrt{\lambda}) \sqrt{K}} + \tfrac{ \mathbb{E}\left[V_{\alpha} \left( \x_{\ell} \right)\right] }{ \gamma_0 (1 - \sqrt{\lambda}) \sqrt{K} }.
\end{align*} 
Observing that $1 - \sqrt{\lambda} \geq \tfrac{1}{2}$, we have %and that $\gamma_0 < \uvs{\tfrac{1}{L_0^{V_{\alpha}}}}$, we have
	\begin{align*}
			\mathbb{E}\left[\|G_{1/\gamma_0}(\x_{R_{\ell,K}})\|{^2}\right]   & \leq \uvs{\tfrac{1}{ \sqrt{K}}} \left( {\uvs{8}(1-\ln(\lambda))\rho}  +  { \uvs{8}\mu \sum_{k = \ell}^{{K-1}} \gamma_k \epsilon_{k}} + \tfrac{2 \mathbb{E}\left[V_{\alpha} \left( \x_{\ell} \right)\right] }{\uvs{\gamma_0}} \right),
\end{align*} 
which is the desired result.
}

\noindent (ii) We now consider the summation $\sum_{k=\ell}^{K} \gamma_k \epsilon_k$. {Since $\gamma_k = \tfrac{\gamma_0}{\sqrt{k}}$, $\uvs{\gamma_0 = \tfrac{1}{2L^{V_{\alpha}}_1}}$, and $\epsilon_k = \tfrac{p}{\sqrt{k}}$, we have that	$\sum_{k=\ell}^{{K-1}} \gamma_k \epsilon_k = \sum_{k=\ell}^{{K-1}} \tfrac{p\gamma_0}{k} \le p\gamma_0 (1-\ln(\lambda))$. Consequently,}
\begin{align*}
	%\mathbb{E}\left[\|G_{1/\gamma_0}(\x_{R_{\ell,K}})\|\right]^2   & \leq   \mathcal{O}\left(\tfrac{\ln(K)}{\sqrt{K}}\right) \\
	\mathbb{E}\left[\|G_{1/\gamma_0}(\x_{R_{\ell,K}})\|\right]   & \leq  \sqrt{ \mathcal{O}\left(\tfrac{1}{\sqrt{K}}\right)}.
\end{align*}
{This implies $K_{\varepsilon} = \mathcal{O}(\epsilon^{-4})$ and proceeding as earlier}, 
\begin{align*}
	\sum_{k=0}^{{K_{\varepsilon}-1}} M_k = \sum_{k=0}^{{K_{\varepsilon}-1}} \lceil 1 + a {\sqrt{k}} \rceil \leq \mathcal{O} \left(K_{\varepsilon}\right) + \mathcal{O} \left(a K_{\varepsilon}^{3/2}\right)  = \mathcal{O} \left({\varepsilon^{-6}}\right). 
%\left(\varepsilon^{-4}{\ln^2}(\varepsilon^{-1})\right).
\end{align*}
\end{proof}
\section{Concluding remarks}
{The computation of Nash equilibria has been a question of interest over the last 70 years. More recently, there has been a focus on resolving such problems when player problems are complicated by the presence of uncertainty. Yet, most advances in such regimes have necessitated monotonicity of the concatenated gradient map or a suitable potentiality requirement. In this paper, we consider an optimization-based approach reliant on the {\bf NI} function, requiring the satisfaction of a suitable regularity condition. Our proposed sampling-enabled inexact gradient framework is equipped with rate and complexity guarantees, both of which are novel in this context. Our future efforts will consider whether the required regularity condition holds in more general settings. In addition, we intend to examine whether such avenues can contend with the computation of local or quasi-Nash equilibria in uncertain settings. }

\bibliographystyle{plain}
\bibliography{demobib-v1,wsc11-v03a,ref_paIRIG_v01_fy}

\begin{thebibliography}{10}

\bibitem{beck14introduction}
A.~Beck.
\newblock {\em Introduction to nonlinear optimization}, volume~19 of {\em
  MOS-SIAM Series on Optimization}.
\newblock Society for Industrial and Applied Mathematics (SIAM), Philadelphia,
  PA; Mathematical Optimization Society, Philadelphia, PA, 2014.
\newblock Theory, algorithms, and applications with MATLAB.

\bibitem{Borkar08}
V.~S. Borkar.
\newblock {\em Stochastic Approximation: A Dynamical Systems Viewpoint}.
\newblock Cambridge University Press, 2008.

\bibitem{CSY2021MPEC}
S.~Cui, U.~V. Shanbhag, and F.~Yousefian.
\newblock Complexity guarantees for an implicit smoothing-enabled method for
  stochastic {MPEC}s.
\newblock {\em Math. Program.}, 198(2):1153--1225, 2023.

\bibitem{dafermos88sensitivity}
S.~Dafermos.
\newblock Sensitivity analysis in variational inequalities.
\newblock {\em Math. Oper. Res.}, 13(3):421--434, 1988.

\bibitem{kanzownonsmooth}
A.~Dreves and C.~Kanzow.
\newblock {Nonsmooth optimization reformulations characterizing all solutions
  of jointly convex generalized {N}ash equilibrium problems}.
\newblock {\em Computational Optimization and Applications}, 50(1):23--48,
  September 2011.

\bibitem{facchinei02finite}
F.~Facchinei and J.-S. Pang.
\newblock {\em Finite-dimensional Variational Inequalities and Complementarity
  Problems. {V}ols. {I,II}}.
\newblock Springer Series in Operations Research. Springer-Verlag, New York,
  2003.

\bibitem{facc2010}
F.~Facchinei and J.~S. Pang.
\newblock {N}ash equilibria: the variational approach.
\newblock In {\em Convex Optimization in Signal Processing and Communications},
  2010.

\bibitem{fudenberg98theory}
D.~Fudenberg and D.~K. Levine.
\newblock {\em The theory of learning in games}, volume~2 of {\em MIT Press
  Series on Economic Learning and Social Evolution}.
\newblock MIT Press, Cambridge, MA, 1998.

\bibitem{fukushima1992}
M.~Fukushima.
\newblock Equivalent differentiable optimization problems and descent methods
  for asymmetric variational inequality problems.
\newblock {\em Mathematical Programming}, 53(1):99--110, Jan 1992.

\bibitem{gurkan}
G.~Gürkan and J.S. Pang.
\newblock Approximations of {N}ash equilibria.
\newblock {\em Math. Program.}, 117:223--253, 03 2009.

\bibitem{xu-SVI}
H.~Jiang and H.~Xu.
\newblock Stochastic approximation approaches to the stochastic variational
  inequality problem.
\newblock {\em IEEE Trans. Automat. Control}, 53(6):1462--1475, 2008.

\bibitem{kannan}
A.~Kannan and {U. V.} Shanbhag.
\newblock The pseudomonotone stochastic variational inequality problem:
  Analytical statements and stochastic extragradient schemes.
\newblock In {\em 2014 American Control Conference, ACC 2014}, Proceedings of
  the American Control Conference, pages 2930--2935, United States, 2014.
  Institute of Electrical and Electronics Engineers Inc.
\newblock 2014 American Control Conference, ACC 2014 ; Conference date:
  04-06-2014 Through 06-06-2014.

\bibitem{kannan-pseudo}
A.~Kannan and U.~V. Shanbhag.
\newblock Optimal stochastic extragradient schemes for pseudomonotone
  stochastic variational inequality problems and their variants.
\newblock {\em Comput. Optim. Appl.}, 74(3):779--820, 2019.

\bibitem{krawczyk}
J.~Krawczyk and S.~Uryasev.
\newblock Relaxation algorithms to find {N}ash equilibria with economic
  applications.
\newblock {\em Environmental Modeling \& Assessment}, 5:63--73, 01 2000.

\bibitem{Kush03}
H.~J. Kushner and G.~G. Yin.
\newblock {\em Stochastic Approximation and Recursive Algorithms and
  Applications}.
\newblock Springer, New York, 2003.

\bibitem{lei2020asynchronous}
J.~Lei and U.~V. Shanbhag.
\newblock Asynchronous variance-reduced block schemes for composite non-convex
  stochastic optimization: block-specific steplengths and adapted batch-sizes.
\newblock {\em Optim. Methods Softw.}, 37(1):264--294, 2022.

\bibitem{lei-distributed}
J.~Lei and U.~V. Shanbhag.
\newblock Distributed variable sample-size gradient-response and best-response
  schemes for stochastic {N}ash equilibrium problems.
\newblock {\em SIAM Journal on Optimization}, 32(2):573--603, 2022.

\bibitem{lei_inexact}
J.~Lei, U.~V. Shanbhag, J.S. Pang, and S.~Sen.
\newblock On synchronous, asynchronous, and randomized best-response schemes
  for stochastic {N}ash games.
\newblock {\em Mathematics of Operations Research}, 45(1):157--190, 2020.

\bibitem{Mastroeni2003}
G.~Mastroeni.
\newblock Gap functions for equilibrium problems.
\newblock {\em Journal of Global Optimization}, 27(4):411--426, Dec 2003.

\bibitem{nash50equilibrium}
J.~F. Nash, Jr.
\newblock Equilibrium points in {$n$}-person games.
\newblock {\em Proc. Nat. Acad. Sci. U. S. A.}, 36:48--49, 1950.

\bibitem{nikaido1955}
H.~Nikaid{\^o} and K~Isoda.
\newblock Note on non-cooperative convex game.
\newblock {\em Pacific Journal of Mathematics}, 5:807--815, 1955.

\bibitem{parise-distributed}
F.~Parise, S.~Grammatico, B.~Gentile, and J.~Lygeros.
\newblock Distributed convergence to {N}ash equilibria in network and average
  aggregative games.
\newblock {\em Automatica J. IFAC}, 117:108959, 9, 2020.

\bibitem{raghunathan}
A.~U. Raghunathan, A.~Cherian, and D.~K. Jha.
\newblock Game theoretic optimization via gradient-based {N}ikaido-{I}soda
  function.
\newblock In {\em Proceedings of the 36th International Conference on Machine
  Learning, {ICML} 2019, 9-15 June 2019, Long Beach, California, {USA}},
  volume~97 of {\em Proceedings of Machine Learning Research}, pages
  5291--5300. {PMLR}, 2019.

\bibitem{shapiro09lectures}
A.~Shapiro, D.~Dentcheva, and A.~Ruszczy{\'n}ski.
\newblock {\em Lectures on stochastic programming}, volume~9 of {\em MPS/SIAM
  Series on Optimization}.
\newblock SIAM, Philadelphia, PA, 2009.
\newblock Modeling and theory.

\bibitem{shapiro-sa}
A.~Shapiro and H.~Xu.
\newblock Stochastic mathematical programs with equilibrium constraints,
  modeling and sample average approximation.
\newblock {\em Optimization}, 57:395--418, 04 2008.

\bibitem{kanzow}
A.~von Heusinger and C.~Kanzow.
\newblock Optimization reformulations of the generalized {N}ash equilibrium
  problem using {N}ikaido-{I}soda-type functions.
\newblock {\em Comput. Optim. Appl.}, 43(3):353--377, 2009.

\bibitem{farzad-SVI}
F.~Yousefian, A.~Nedi\'{c}, and U.~V. Shanbhag.
\newblock On smoothing, regularization, and averaging in stochastic
  approximation methods for stochastic variational inequality problems.
\newblock {\em Math. Program.}, 165(1):391--431, 2017.

\end{thebibliography}

\end{document}